\documentclass{amsart}
\usepackage{amssymb}
\usepackage{mathrsfs}
\usepackage[all]{xy}
\usepackage{graphicx}
\usepackage{tikz}
\usetikzlibrary{cd}

\newcommand{\Z}{\mathbb{Z}}
\newcommand{\Q}{\mathbb{Q}}
\newcommand{\C}{\mathbb{C}}
\newcommand{\Qlbar}{\overline{\Q_\ell}}
\newcommand{\bk}{\Bbbk}

\DeclareMathOperator{\id}{id}
\DeclareMathOperator{\im}{im}
\DeclareMathOperator{\Hom}{Hom}
\DeclareMathOperator{\End}{End}
\DeclareMathOperator{\Ext}{Ext}
\newcommand{\simto}{\overset{\sim}{\to}}
\newcommand{\Ob}{\mathrm{Ob}}
\DeclareMathOperator{\diag}{diag}

\newcommand{\fg}{\mathfrak{g}}
\newcommand{\fb}{\mathfrak{b}}
\newcommand{\fh}{\mathfrak{h}}
\newcommand{\cO}{\mathcal{O}}
\newcommand{\gr}{\mathrm{gr}}

\newcommand{\ngmod}{\mathsf{mod}}
\newcommand{\gmod}{\mathsf{gmod}}
\newcommand{\lh}{\text{-}}
\newcommand{\opp}{\mathrm{opp}}
\DeclareMathOperator{\gdim}{gdim}
\DeclareMathOperator{\Sym}{Sym}
\newcommand{\fm}{\mathfrak{m}}
\newcommand{\parity}{\mathrm{parity}}
\newcommand{\proj}{\mathrm{proj}}

\newcommand{\ux}{{\underline{x}}}
\newcommand{\uy}{{\underline{y}}}
\newcommand{\uw}{{\underline{w}}}
\newcommand{\uH}{\underline{H}}
\DeclareMathOperator{\ch}{ch}

\newcommand{\geom}{\mathrm{geom}}
\newcommand{\SBim}{\mathsf{SBim}}
\newcommand{\oSBim}{\overline{\SBim}}
\newcommand{\Parity}{\mathsf{Parity}}
\newcommand{\cB}{\mathcal{B}}
\newcommand{\BS}{\mathrm{BS}}
\newcommand{\oB}{\overline{B}}

\newcommand{\BsR}{
 \begin{tikzpicture}[thick,baseline=-2pt,scale=0.05]
  \draw (0,-5) to (0,0);
  \node at (0,0) {\small $\bullet$};
 \end{tikzpicture}
}

\newcommand{\RBs}{
 \begin{tikzpicture}[thick,baseline=-2pt,scale=0.05]
  \node at (0,0) {\small $\bullet$};
  \draw (0,0) to (0,5);
 \end{tikzpicture}
}
\newcommand{\BsBs}{
 \begin{tikzpicture}[thick,baseline=-2pt,scale=0.05]
  \draw (0,-5) to (0,5);
  \node at (0,0) {}; 
 \end{tikzpicture}
}
\newcommand{\BsRBs}{
 \begin{tikzpicture}[thick,baseline=-2pt,scale=0.05]
  \draw (0,-5) to (0,-1.7);
  \node at (0,-1.7) {\small $\bullet$};
  \node at (0,1.7) {\small $\bullet$};
  \draw (0,1.7) to (0,5);
 \end{tikzpicture}
}

\newcommand{\Kb}{K^{\mathrm{b}}}
\newcommand{\Db}{D^{\mathrm{b}}}

\newcommand{\Dbm}{D^\mathrm{b}_\mathrm{m}}
\newcommand{\mix}{\mathrm{mix}}
\newcommand{\Dmix}{D^\mix}
\newcommand{\Pmix}{\mathsf{P}^\mix}

\newcommand{\Tmix}{\mathsf{Tilt}^\mix}
\newcommand{\scS}{\mathscr{S}}
\newcommand{\IC}{\mathcal{IC}}
\newcommand{\cP}{\mathcal{P}}
\newcommand{\cI}{\mathcal{I}}
\newcommand{\cT}{\mathcal{T}}

\newcommand{\pt}{\mathrm{pt}}
\newcommand{\BGB}{B\backslash G/B}
\newcommand{\UGB}{U\backslash G/B}
\newcommand{\BGU}{B\backslash G/U}
\newcommand{\UGU}{U\backslash G/U}
\newcommand{\BGUvee}{B^\vee \backslash G^\vee/U^\vee}
\newcommand{\BGPs}{B\backslash G/P^s}
\newcommand{\UGPs}{U\backslash G/P^s}
\newcommand{\BGPt}{B\backslash G/P^t}
\newcommand{\UGPt}{U\backslash G/P^t}

\newcommand{\hDbm}{\widehat{D}^\mathrm{b}_\mathrm{m}}
\makeatletter
\newcommand*\leftdash{\rotatebox[origin=c]{-45}{$\dabar@\dabar@\dabar@$}}
\newcommand*\rightdash{\rotatebox[origin=c]{45}{$\dabar@\dabar@\dabar@$}}
\makeatother
\newcommand{\UGUby}{B\leftdash G\rightdash B}

\newcommand{\ubk}{\underline{\bk}}
\newcommand{\cE}{\mathcal{E}}
\newcommand{\cF}{\mathcal{F}}
\newcommand{\cG}{\mathcal{G}}
\newcommand{\cH}{\mathcal{H}}

\newcommand{\For}{\mathsf{For}}
\newcommand{\la}{\langle}
\newcommand{\ra}{\rangle}
\newcommand{\bH}{\mathbb{H}}
\newcommand{\bV}{\mathbb{V}}

\newcommand{\can}{\mathsf{can}}

\newcommand{\cA}{\mathcal{A}}
\newcommand{\sC}{\mathsf{C}}
\DeclareMathOperator{\HOM}{HOM}
\DeclareMathOperator{\END}{END}
\newcommand{\cHom}{\mathcal{H}\mathrm{om}}
\newcommand{\cHOM}{\mathcal{H}\mathrm{OM}}
\newcommand{\Ch}{\mathrm{Ch}}
\newcommand{\Chb}{\mathrm{Ch}^\mathrm{b}}
\newcommand{\PChb}{\mathrm{PCh}^\mathrm{b}}
\newcommand{\PKb}{\mathrm{P}K^\mathrm{b}}
\newcommand{\z}{{}^0\!}
\newcommand{\p}{{}^p\!}

\renewcommand{\ddots}{\rotatebox{-20}{$\cdots$}}

\newtheorem*{thm*}{Theorem}
\numberwithin{equation}{section}
\newtheorem{thm}{Theorem}[section]
\newtheorem{lem}[thm]{Lemma}
\newtheorem{prop}[thm]{Proposition}
\newtheorem{cor}[thm]{Corollary}
\theoremstyle{definition}
\newtheorem{defn}[thm]{Definition}
\theoremstyle{remark}
\newtheorem{rmk}[thm]{Remark}
\newtheorem{ex}[thm]{Example}

\title{Modular Koszul duality for Soergel bimodules}
\author{Shotaro Makisumi}
\date{April 3, 2020}
\address{Department of Mathematics, Columbia University, New York, NY, U.S.A.}
\email{makisum@math.columbia.edu}

\begin{document}

\begin{abstract}
 We generalize the modular Koszul duality of Achar--Riche~\cite{AR-II} to the setting of Soergel bimodules associated to any finite Coxeter system. The key new tools are a functorial monodromy action and wall-crossing functors in the mixed modular derived category of ibid. In characteristic 0, this duality together with Soergel's conjecture (proved by Elias--Williamson~\cite{EW-hodge}) imply that our Soergel-theoretic graded category $\cO$ is Koszul self-dual, generalizing the result of Beilinson--Ginzburg--Soergel~\cite{Soe90, BGS}.
\end{abstract}

\maketitle

\section{Introduction}

Let $\fg$ be a semisimple complex Lie algebra. Fix a Borel subalgebra $\fb$ containing a Cartan subalgebra $\fh$, and consider the principal block $\cO_0$ of the associated BGG category $\cO$, i.e.~the block containing the trivial representation. As is well known, $\cO_0$ is a finite-length abelian category with enough projectives. It therefore has a minimal projective generator $P$, so that $\cO_0$ is equivalent to the category of finitely-generated modules over
\[
 A := \End_{\cO_0}(P)^\opp.
\]

Beilinson--Ginzburg--Soergel~\cite{BGS} showed that this ring admits a \emph{Koszul grading}: a positive grading $A = \bigoplus_{i \ge 0} A_i$ with $A_0$ semisimple, and such that the left $A$-module $A_0 = A/A_{>0}$ admits a graded projective resolution $P^\bullet \to A_0$ where $P^i$ is generated in degree $i$. The existence of a Koszul grading on the algebra controlling $\cO_0$ is a deep fact closely related to the Kazhdan--Lusztig conjecture.

For a Koszul ring $B = \bigoplus_{i \ge 0} B_i$, its Koszul dual ring is $E(B) := \Ext_{B\lh\ngmod}^\bullet(B_0)$, where Ext is taken in ungraded $B$-modules. It was moreover shown in \cite{Soe90, BGS} that $A$ is Koszul self-dual: $A \cong E(A)$ as graded algebras. Thus the Koszul grading on $A$ reveals a hidden self-duality of $\cO_0$.

The Weyl group $W$ of $\fg$ acts on the Cartan subalgebra $\fh$. By work of Soergel, the algebra $A$ only depends on this $W$-representation. The theory of Soergel bimodules \cite{Soe92, Soe00, Soe07} allows one more generally to define a graded algebra $A_{W, \fh}$ for any suitable ``reflection faithful realization'' $\fh$ of an arbitrary Coxeter system $W$. See \S\ref{sec:main-result} for background on Soergel bimodules as well as the definition of $A_{W, \fh}$.

For Soergel bimodules, one still has an analogue of the Kazhdan--Lusztig conjecture known as Soergel's conjecture. Elias--Williamson suggested \cite[Remark~3.4]{EW-soergel-calculus} that there should also be a rich Koszul duality in this generality. In this paper, we realize this vision for finite Coxeter systems.

The following theorem is a consequence of our main result together with the Soergel's conjecture proved by Elias--Williamson \cite{EW-hodge}.
\begin{thm} \label{thm:ksd-geom}
 For any finite (not necessarily crystallographic) Coxeter group $W$ and its geometric representation $\fh_\geom$, the graded algebra $A_{W, \fh_\geom}$ is Koszul self-dual.
\end{thm}

\subsection{Mixed modular derived category}
The Koszul duality of \cite{BGS} is a derived equivalence involving $\cO_0^\gr$ (``graded category $\cO_0$''), a graded version of $\cO_0$. As a first step towards Theorem~\ref{thm:ksd-geom}, we defined in \cite{Mak-moment} a Soergel-theoretic analogue of $\cO_0^\gr$ for a general pair $(W, \fh)$ using the mixed modular derived category formalism of~\cite{AR-II}.

As explained in~\cite{BGS}, the grading on $\cO_0$ comes from mixed geometry; $\cO_0^\gr$ may be identified with a certain category of mixed $\ell$-adic perverse sheaves on the flag variety defined over a finite field. In~\cite{AR-II}, Achar--Riche introduced a new approach to defining the mixed category based on the parity complexes of~\cite{JMW}. This ``mixed (modular) derived category'' has a notion of weights and Tate twist, and for positive characteristic coefficients, may serve as a replacement for mixed $\ell$-adic sheaves. In particular, for a connected complex reductive group $G$ with a Borel subgroup $B$,~\cite{AR-II} defined and studied the categories of $B$-constructible ``mixed complexes'' and ``mixed perverse sheaves'' on the flag variety $G/B$ with coefficients in a field $\bk$:
\[
 \Dmix_{(B)}(G/B, \bk) \supset \Pmix_{(B)}(G/B, \bk).
\]

Parity complexes on the flag variety are related to Soergel (bi)modules (see \S\ref{ss:relation-to-parity}). Motivated by this connection, in~\cite{Mak-moment} we adapted the Achar--Riche mixed derived category to the setting of Soergel bimodules associated to $(W, \fh)$, not necessarily arising from complex reductive groups; this is our analogue of $\cO_0^\gr$. We will recall the facts we need from this framework in~\S\ref{ss:mixed-modular}.

\subsection{Modular Koszul duality}
Geometrically, the Koszul duality of $\cO_0$ is a derived equivalence relating mixed sheaves on Langlands dual flag varieties. Achar--Riche~\cite{AR-II} proved an analogous equivalence for the mixed modular derived category, which they call ``modular Koszul duality,'' although it may not involve any Koszul algebra.

Our main result (Theorem~\ref{thm:mkd}) is an analogous result in the setting of Soergel theory for any finite Coxeter system. It should be thought of as a derived equivalence relating mixed modular sheaves on possibly non-existent Langlands dual flag varieties.

Our key new tools are a monodromy action and wall-crossing functors in our analogue of the mixed derived category. These tools allow us to imitate the strategy of the characteristic-zero Koszul duality of Bezrukavnikov--Yun~\cite{BY}. In particular, even in the geometric setting of~\cite{AR-II}, our approach gives a new proof of modular Koszul duality that is independent of~\cite{BY} and of the Kazhdan--Lusztig conjecture.

In~\cite{AR}, Achar--Riche identified $\Pmix_{(B)}(G/B, \C)$ with the $\cO_0^\gr$ of~\cite{BGS}.\footnote{This, however, relies on the Koszulity of $\cO_0^\gr$, hence on some form of the Kazhdan--Lusztig conjecture.} Via this identification, our approach also gives a new proof of the classical Koszul duality of~\cite{Soe90, BGS}.

\begin{rmk}
 In \cite{Mak-moment}, we crucially used the Braden--MacPherson and Fiebig theory of moment graph sheaves. However, once the framework is set up, all results we quote from ibid.~are exact analogues of those of~\cite{AR-II}, which in turn are analogues of well-known results in characteristic 0. The constructions we introduce are already new for $\Dmix_{(B)}(G/B, \C)$, and should be accessible to readers who prefer to think in this setting.
\end{rmk}

\subsection{Related work}
For finite dihedral groups, the Koszulity and Koszul self-duality of $A_{W, \fh_\geom}$ was proved earlier by explicit methods by Sauerwein~\cite{Sau}.

After an early draft of this article had been written in 2015, the author learned of a project by Achar, Riche, and Williamson that contained constructions similar to those of this article. Our joint work \cite{AMRW} (in particular) clarifies and extends the constructions in \S\ref{sec:monodromy}--\ref{sec:wall-crossing}.

\subsection{Contents}
In~\S\ref{sec:main-result}, we recall some background on Soergel (bi)modules and the mixed modular derived category and state the main result. After some preliminaries in~\S\ref{sec:preliminaries}, we introduce the key new constructions in~\S\ref{sec:monodromy} (monodromy action) and~\S\ref{sec:wall-crossing} (wall-crossing functors). The main result is proved in~\S\ref{sec:koszul-duality}.

\subsection{Acknowledgements}
This article is an edited version of one part of the author's 2017 Ph.D.~thesis at Stanford University. It is a pleasure to thank my advisor, Z.~Yun, for his continued guidance and support---for encouraging and supporting my travels through his Packard Foundation fellowship, and for countless hours of helpful conversation, mathematical or otherwise, throughout my time as a graduate student.

I thank B.~Elias for his 2014 WARTHOG lectures on Soergel bimodules at the University of Oregon. I owe much to P.~Fiebig for his advice and encouragement. The wall-crossing functor of~\S\ref{sec:wall-crossing} has its origins in my discussions with him during a visit to Universit\"at Erlangen--N\"urnberg in September 2015. The collaboration~\cite{AMRW} resulted from a workshop at the American Institute for Mathematics in March 2016. I thank AIM for hosting the workshop, and P.~Achar, S.~Riche, and G.~Williamson for many enlightening discussions as well as comments on a preliminary draft. I also thank G.~Dhillon for valuable comments on the first arXiv version.

\section{Background and main result}
\label{sec:main-result}

\subsection{Background on the Hecke algebras}
\label{ss:coxeter}
Let $(W, S)$ be a Coxeter system. We always tacitly assume that $|S| < \infty$. We denote by $e$ the identity element, $\ell\colon W \to \Z_{\ge0}$ the length function, and $\le$ the Bruhat order. An \emph{expression} is a word $\uw = s_1 \cdots s_k$ in $S$. We write $w = s_1 \cdots s_k$ for the corresponding element in $W$.

We follow Soergel's normalization for the Hecke algebra; see~\cite{Soe90} for details. The Hecke algebra $\cH_W$ is the algebra with free $\Z[v, v^{-1}]$-basis $\{ H_w \}_{w \in W}$ and multiplication
\[
 H_wH_s = \begin{cases}
           H_{ws} &\mbox{if } ws > w, \\
           (v^{-1} - v)H_w + H_{ws} &\mbox{if } ws < w.
          \end{cases}
\]
This algebra has another basis, the Kazhdan--Lusztig basis $\{ \uH_w \}_{w \in W}$. In our normalization, $\uH_s = H_s + vH_e$ for $s \in S$. For any expression $\uw = s_1 \cdots s_k$, we set $\uH_\uw := \uH_{s_1} \cdots \uH_{s_k}$.

\subsection{Background on Soergel (bi)modules}
\label{ss:sbim}
A \emph{realization} of $(W, S)$ over a commutative ring $\bk$ \cite{Eli-dihedral, EW-soergel-calculus} is a triple
\[
 (\fh, \{\alpha_s^\vee\}_{s \in S} \subset \fh, \{\alpha_s\}_{s \in S} \subset \fh^*),
\]
where $\fh$ is a finite-rank free $\bk$-module and $\fh^* = \Hom_\bk(\fh, \bk)$, such that: $\la \alpha_s^\vee, \alpha_s \ra = 2$ for all $s \in S$; the assignment $s(v) := v - \la v, \alpha_s \ra \alpha_s^\vee$ for $s \in S, v \in \fh$ defines a representation of $W$ on $\fh$; and a technical condition \cite[(3.3)]{EW-soergel-calculus} that will always holds in our setting. We often simply speak of a realization $\fh$.

We assume that $\bk$ is a field of characteristic not equal to 2, and that the $W$-representation $\fh$ is \emph{reflection faithful} in the sense of \cite[Definition~1.5]{Soe07}:\footnote{Soergel assumed in addition that $\bk$ is infinite. However, as has been noted for example in \cite[\S1.3]{Ric}, this assumption is not necessary and was only imposed in \cite{Soe07} in order to identify $R = \Sym_\bk(\fh^*)$ with the ring of regular functions on $\fh$.} $\fh$ is faithful, and for all $w \in W$, the fixed subspace $\fh^w$ has codimension 1 if and only if $w$ is a reflection in $W$, i.e.~a conjugate of an element of $S$. These conditions guarantee that Soergel's theory in \cite{Soe07} is available.

We call the data $(W, \fh)$ a \emph{reflection faithful realization}. This is the starting data for our categories.

\begin{rmk}
 Reflection faithfulness is a serious condition. Note that an infinite Coxeter system admits no faithful representation over a finite field or its algebraic closure. For an irreducible finite Coxeter system, the only reflection faithful representations over $\mathbb{R}$ are Galois conjugates of the geometric representation (see \cite[Theorem~1.2]{MM} for a more general statement).
\end{rmk}
\begin{ex} \label{ex:cartan-realization}
 Let $G$ be a connected reductive group, and choose a Borel subgroup $B$ containing a maximal torus $T$. The associated Weyl group $W$ has the natural structure of a Coxeter system. Let $X_*(T)$ be group of cocharacters of $T$. Then for any field $\bk$, the base change $\bk \otimes_\Z X_*(T)$ has the natural structure of a realization of $W$. In this way, the triple $(G,B,T)$ gives rise to a realization $(W,\fh)$, called the \emph{Cartan realization}, over any field $\bk$. Libedinsky shows in \cite[Appendix]{Lib} that the Cartan realization is reflection faithful if $\operatorname{char} \bk \notin \{2, 3\}$.
\end{ex}

Let $R = \Sym_\bk(\fh^*)$ be the symmetric algebra, viewed as a ($\Z$-)graded algebra with $\deg \fh^* = 2$. Let $R\lh\gmod\lh R$ be the category of ($\Z$-)graded $R$-bimodules and graded $R$-bimodule homomorphisms (of degree 0). This category has a grading shift autoequivalence $\{1\}$, defined on $M = \bigoplus_{i \in \Z} M_i$ by $(M\{n\})_i = M_{n+i}$. It is also naturally monoidal with the product $\otimes_R$, which we often omit from the notation.

The $W$-action on $\fh$ induces a $W$-action on $R$. For $s \in S$, let $R^s$ denote the ring of $s$-invariants in $R$. Define the graded $R$-bimodule
\[
 B_s := R \otimes_{R^s} R\{1\}
\]
and the endofunctor
\[
 \theta_s = B_s \otimes_R (-)\colon R\lh\gmod\lh R \to R\lh\gmod\lh R.
\]
For any expression $\uw = s_1 \cdots s_k$, define the \emph{Bott--Samelson bimodule}
\[
 B_\uw := \theta_{s_1} \cdots \theta_{s_k}(R) = B_{s_1} \cdots B_{s_k},
\]
where $R$ is the regular $R$-bimodule. (Note that $B_\varnothing = R$ for the empty expression $\varnothing$.) Let $\SBim_\BS(W, \fh)$ be the smallest strictly full subcategory of $R\lh\gmod\lh R$ containing all Bott--Samelson bimodules and closed under $\oplus$ and $\{n\}$. Let $\SBim(W, \fh)$ be the Karoubi envelope of $\SBim_\BS(W, \fh)$, which we always identify with a strictly full subcategory of $R\lh\gmod\lh R$:
\begin{equation} \label{eq:sbim}
 \SBim_\BS(W, \fh) \subset \SBim(W, \fh) \subset R\lh\gmod\lh R.
\end{equation}
Each of these categories is monoidal under $\otimes_R$ with unit object $R$. The objects of $\SBim(W, \fh)$ are called \emph{Soergel bimodules}. Soergel~\cite{Soe07} proved the following classification of the indecomposable objects.
\begin{prop} \label{prop:sbim-bs-characterization}
 For each $w \in W$, there is an object $B_w \in \SBim(W, \fh)$, characterized up to isomorphism by the following property: for any reduced expression $\uw$ of $w$, $B_w$ is the unique indecomposable direct summand of $B_\uw$ that does not occur as a direct summand of $B_\ux$ for any expression $\ux$ with $\ell(\ux) < \ell(\uw)$.
 
 The set $\{ B_w \}_{w \in W}$ is a complete list of isomorphism classes of indecomposable Soergel bimodules up to shift. Every object of $\SBim(W, \fh)$ is isomorphic to a finite direct sum of shifts of various $B_w$, and such a decomposition is unique in the obvious sense.
\end{prop}
The split Grothendieck group $[\SBim(W, \fh)]$ of $\SBim(W, \fh)$ is a $\Z[v, v^{-1}]$-algebra under $v[M] = [M\{1\}]$ and $[M][N] = [M \otimes_R N]$. Soergel's categorification theorem \cite{Soe07} states that the assignment $\uH_s \mapsto [B_s]$ for $s \in S$ determines a $\Z[v, v^{-1}]$-algebra isomorphism $\cH_W \simto [\SBim(W, \fh)]$. Let $\ch: [\SBim(W,\fh)] \simto \cH_W$ denote the inverse isomorphism.
\begin{defn} \label{defn:soergel-conjecture}
 We say that the realization $(W, \fh)$ satisfies \emph{Soergel's conjecture}\footnote{Soergel only conjectured this result for specific realizations in characteristic 0.} if $\ch([B_w]) = \uH_w$ for all $w \in W$.
\end{defn}
This is a Soergel-theoretic analogue of the Kazhdan--Lusztig conjecture.

Let $\gmod\lh R$ denote the category of graded right $R$-modules and graded $R$-module homomorphisms, and let $\{1\}$ denote as before the grading shift down. Let $\fm \subset R$ denote the augmentation ideal, i.e.~the graded ideal of $R$ generated by $\fh^*$. Consider the functor
\[
 R\lh\gmod\lh R \rightarrow \gmod\lh R: M \mapsto \bk \otimes_R M,
\]
where $\bk := R/\fm$. That is, $\bk \otimes_R M$ is obtained from $M$ by killing the image of positive degree elements of $R$ acting on the left. Now, define the categories
\[
 \oSBim_\BS(W, \fh) \subset \oSBim(W, \fh) \subset \gmod\lh R
\]
to be the essential images of the categories in~\eqref{eq:sbim} under the functor above.

The objects of $\oSBim(W, \fh)$ are called (right) \emph{Soergel modules}. The modules $\oB_w := \bk \otimes_R B_w$ remain indecomposable (as follows from Proposition~\ref{prop:sbim-equivariant-formality} below) and pairwise distinct. Thus $\oB_w$ is again characterized as the ``largest'' direct summand of a \emph{Bott--Samelson module} $\oB_\uw := \bk \otimes_R B_\uw$, and we have a classification theorem entirely analogous to the case of bimodules. We similarly define Bott--Samelson and Soergel modules in $R\lh\gmod$ by reducing the right $R$-action.

We can now define the graded algebra $A_{W, \fh}$ from the introduction.
\begin{defn} \label{defn:A}
 Let $A_{W,\fh}$ be the graded endomorphism algebra
 \[
  A_{W, \fh} := \bigoplus_{n \in \Z} \Hom_{\gmod\lh R}(\bk \otimes_R B, \bk \otimes_R B\{n\}),
 \]
 where $B := \bigoplus_{w \in W} B_w \in \SBim(W, \fh)$.
\end{defn}

\subsection{Relation to parity sheaves}
\label{ss:relation-to-parity}
Parity complexes were introduced by Juteau--Mautner--Williamson \cite{JMW}. In this subsection, we recall some well-known results (essentially due to Soergel \cite{Soe90, Soe00}) relating parity complexes on flag varieties and Soergel (bi)modules. Since we never actually use parity complexes, we will be brief; their importance for us is purely as motivation.

Let $G$ be a connected complex reductive group. Choose a Borel subgroup $B$ containing a maximal torus $T$. Then $B$ act by left multiplication on the flag variety $G/B$; the $B$-orbits are the Schubert cells $X_w$, $w \in W$, where $W$ is the Weyl group. Let $\bk$ be a field, and let $\Db_B(G/B, \bk)$ (resp.~$\Db_{(B)}(G/B, \bk)$) denote the $B$-equivariant (resp.~$B$-constructible) derived category of sheaves of $\bk$-vector spaces on $G/B$. Denote by $\{1\}$ the cohomological shift (usually denoted by $[1]$).

Let $\Parity_B(G/B, \bk) \subset \Db_B(G/B, \bk)$ be the full additive subcategory of $B$-equi\-variant parity complexes. This category is stable under $\{1\}$, and indecomposable objects up to shift and isomorphism are indexed by $W$. More specifically, for $w \in W$, the indecomposable parity complex $\cE_w$, called \emph{parity sheaf}, is characterized by having support $\overline{X_w}$ and the normalization $\cE_w|_{X_w} \cong \underline{\bk}_{X_w}\{\ell(w)\}$.

The category $\Db_B(G/B, \bk)$ is monoidal under $B$-convolution $\ast$ and unit object $\cE_e = \delta$, the skyscraper at the point stratum $X_e$, and $\Parity_B(G/B, \bk)$ is a monoidal subcategory. For any expression $\uw = s_1 \cdots s_k$, define the \emph{Bott--Samelson parity complex}
\[
 \cE_\uw := \cE_{s_1} \ast \cdots \ast \cE_{s_k}.
\]
(For the empty expression, $\cE_\varnothing = \cE_e = \delta$.) Then $\cE_w$ also admits a ``Bott--Samelson-type'' characterization: for any reduced expression $\uw$ of $w$, it is the unique direct summand of $\cE_\uw$ that does not appear as a direct summand of $\cE_\ux$ for any expression $\ux$ with $\ell(\ux) < \ell(\uw)$.

Let $\Parity^\BS_B(G/B, \bk)$ be the smallest strictly full subcategory of $\Db_B(G/B, \bk)$ containing all Bott--Samelson complexes and closed under $\oplus$ and $\{n\}$. By the preceding discussion, its Karoubi envelope can be identified with $\Parity_B(G/B, \bk)$. Thus we have categories
\begin{equation} \label{eq:parity}
 \Parity^\BS_B(G/B, \bk) \subset \Parity_B(G/B, \bk) \subset \Db_B(G/B, \bk),
\end{equation}
each monoidal under $\ast$ with unit object $\delta$.

Let
\[
 \Parity^\BS_{(B)}(G/B, \bk) \subset \Parity_{(B)}(G/B, \bk) \subset \Db_{(B)}(G/B, \bk)
\]
be the essential images of the categories in~\eqref{eq:parity} under the forgetful functor
\[
 \For\colon \Db_B(G/B, \bk) \to \Db_{(B)}(G/B, \bk).
\]
Then $\Parity_{(B)}(G/B, \bk)$ agrees with the category of $B$-constructible parity complexes on $G/B$. Each $\cE_w := \For(\cE_w)$ remains indecomposable and again admits two characterizations: by a support condition and a normalization, and as the ``largest'' direct summand of the Bott--Samelson parity complex $\cE_\uw := \For(\cE_\uw)$. We also have a classification theorem entirely analogous to the $B$-equivariant case.

The connection to Soergel (bi)modules is as follows. Let $(W, \fh)$ be the Cartan realization over $\bk$ associated to $(G, B, T)$ (see Example~\ref{ex:cartan-realization}). Via the Borel isomorphism $\bH_B^\bullet(\pt, \bk) \cong R = \Sym_\bk(\fh^*)$, total hypercohomology can be viewed as functors
\[
 \bH_B^\bullet\colon \Db_B(G/B, \bk) \to R\lh\gmod\lh R, \qquad \bH^\bullet\colon \Db_{(B)}(G/B, \bk) \to \gmod\lh R
\]
intertwining $\{1\}$. If the characteristic of $\bk$ is good for $G$ and moreover not equal to 2, then it can be deduced from~\cite[\S4]{AR-I} that these functors restrict to equivalences
\[
 (\Parity_B(G/B, \bk), \ast) \simto (\SBim(W, \fh), \otimes_R), \qquad \Parity_{(B)}(G/B, \bk) \simto \oSBim(W, \fh)
\]
sending $\cE_\uw \mapsto B_\uw$, $\cE_w \mapsto B_w$ and $\cE_\uw \mapsto \oB_\uw$, $\cE_w \mapsto \oB_w$, and both intertwining $\{1\}$. These equivalences intertwine $\For$ with $\bk \otimes_R (-)$. Moreover, the first equivalence is monoidal and makes the second equivalence into an equivalence of right module categories.

\subsection{Geometric notation for Soergel (bi)modules}
\label{ss:sbim-notation}
Let $(W, \fh)$ be a reflection faithful realization (over a field $\bk$). Our point of view, motivated by the discussion in~\S\ref{ss:relation-to-parity}, is that a general $(W, \fh)$ should still be thought of as arising from a triple $(G, B, T)$, which however may not actually exist (e.g.~ if $W$ is not crystallographic). Soergel and Bott--Samelson bimodules (resp.~modules) are then $B$-equivariant (resp.~$B$-constructible) parity complexes and Bott--Samelson parity complexes, with $\bk$-coefficients, on the possibly non-existent flag variety $G/B$.

Accordingly, having fixed the realization $(W,\fh)$, we denote the category of Soergel bimodules $\SBim(W, \fh)$ by
\[
 \Parity(\BGB, \bk) \subset R\lh\gmod\lh R,
\]
and the categories of right and left Soergel modules by
\[
 \Parity(\UGB, \bk) \subset \gmod\lh R, \qquad \Parity(\BGU, \bk) \subset R\lh\gmod.
\]
We will often omit the coefficients $\bk$ from the notation. We use similar geometric notation for Bott--Samelson (bi)modules. Moreover, we sometimes write $\ast$ instead of $\otimes_R$ and speak of forgetful functors
\begin{align*}
 \For = (-)\otimes_R \bk &\colon \Parity(\BGB, \bk) \to \Parity(\BGU, \bk), \\
 \For = \bk \otimes_R (-) &\colon \Parity(\BGB, \bk) \to \Parity(\UGB, \bk).
\end{align*}
We also sometimes write $\cE_w$ to mean any of
\[
 B_w \in \Parity(\BGB), \quad \bk \otimes_R B_w \in \Parity(\UGB), \quad B_w \otimes_R \bk \in \Parity(\BGU),
\]
and write $\delta$ for the ``skyscraper'' $\cE_e$.

\begin{rmk}
 We stress that $\BGB$, $\UGB$, $\BGU$ are purely notational device used to emphasize the analogy with geometry. However, the constructions introduced in \S\ref{sec:monodromy}--\S\ref{sec:wall-crossing} of this paper also apply to actual parity complexes in place of Soergel (bi)modules. In that case, the proof of the main theorem in \S\ref{sec:koszul-duality} gives a new proof of the modular Koszul duality of \cite[Theorem~5.4]{AR-II}.
\end{rmk}

\subsection{Background on the mixed modular derived category}
\label{ss:mixed-modular}
In this subsection, we recall the results of~\cite{Mak-moment} adapting the formalism of~\cite{AR-II} to the setting of Soergel (bi)modules.

Let $(W, \fh)$ be a reflection faithful realization. We define the associated \emph{equivariant} and \emph{constructible mixed derived category} by
\[
 \Dmix(\BGB) := \Kb\Parity(\BGB), \qquad \Dmix(\UGB) := \Kb\Parity(\UGB).
\]
Each category has an induced internal grading shift $\{1\}$ and a new cohomological shift $[1]$. Define the \emph{Tate twist} $\la 1 \ra = [1]\{-1\}$. The forgetful functor induces an exact (i.e.~triangulated) functor
\[
 \For\colon \Dmix(\BGB) \to \Dmix(\UGB).
\]
We think of these categories as the $B$-equivariant and $B$-constructible mixed derived categories of a possibly non-existent $G/B$.

The Braden--MacPherson \cite{BMP} and Fiebig \cite{Fie-verma, Fie-coxeter} theory of moment graph sheaves allows us to take this point of view more seriously: it provides a notion of ``strata'' and ``support,'' so that indecomposable Soergel bimodules $\cE_w$ (in the guise of so-called Braden--MacPherson sheaves) may be characterized by a support condition and a normalization analogous to those for parity sheaves. In~\cite{Mak-moment}, we used this theory to define a recollement structure on $\Dmix(\BGB)$ and $\Dmix(\UGB)$, allowing one to speak of the ``standard'' and ``costandard'' sheaves
\[
 \Delta_w := ``i_{w!}\ubk_{X_w}\{\ell(w)\}," \quad \nabla_w := ``i_{w*}\ubk_{X_w}\{\ell(w)\}"
\]
for each $w \in W$, where $i_w$ is the ``inclusion of the Schubert cell $i_w\colon X_w \hookrightarrow G/B$.'' We also defined a ``perverse'' t-structure with hearts\footnote{In~\cite{Mak-moment}, these categories were denoted by $\Pmix(\cB) \subset \Dmix(\cB)$ and $\Pmix_c(\cB) \subset \Dmix_c(\cB)$, where $\cB$ is the Bruhat moment graph associated to $(W, \fh)$.}
\[
 \Pmix(\BGB) \subset \Dmix(\BGB), \qquad \Pmix(\UGB) \subset \Dmix(\UGB),
\]
consisting of \emph{mixed perverse sheaves}, each stable under Tate twist and having simple objects $\{ \IC_w \}_{w \in W}$ up to Tate twist and isomorphism.

One of the main results in~\cite{Mak-moment} is that the pair $(\Pmix(\UGB), \la1\ra)$ has the natural structure of a graded highest weight category\footnote{The notion of a graded highest weight category is as in \cite[Definition~A.1]{AR-II} (who instead use the term ``graded quasihereditary''), except that we do not require the index set to be finite.} indexed by $(W, \le)$ with standard (resp.~costandard) objects $\Delta_w$ (resp.~$\nabla_w$). As in any graded highest weight category, one may then speak of the full additive subcategory
\[
 \Tmix(\UGB) \subset \Pmix(\UGB),
\]
stable under Tate twist, of \emph{tilting objects}. The indecomposable tilting objects are $\{ \cT_w \}_{w \in W}$ up to Tate twist and isomorphism, where $\cT_w$ is characterized by a support condition and a normalization (see \cite[Proposition~A.4]{AR-II}).

When the index set is finite, graded highest weight categories have enough projectives, and enough projectives in the additional presence of a duality functor. Thus for finite $W$, $\Pmix(\UGB)$ contains the usual collection of objects
\[
 \Delta_w, \quad \nabla_w, \quad \IC_w, \quad \cP_w, \quad \cI_w, \quad \cT_w \quad \text{ for } w \in W,
\]
where $\cP_w$ (resp.~$\cI_w$) denotes the projective cover (resp.~injective hull) of $\IC_w$. The category $\Pmix(\UGB)$ is our Soergel-theoretic analogue of $\cO_0^\gr$. However, in general the parity objects $\cE_w$, viewed as a complex supported in cohomological degree 0, need not be perverse.

Consider the objects
\[
 \cP := \bigoplus_{w \in W} \cP_w, \qquad \cE := \bigoplus_{w \in W} \cE_w
\]
in $\Dmix(\UGB)$, and define the graded algebras
\begin{align*}
 A^\proj_{W, \fh} := \left(\bigoplus_{n \in \Z} \Hom(\cP, \cP\la n \ra)\right)^\opp, \qquad A^\parity_{W, \fh} := \bigoplus_{n \in \Z} \Hom(\cE, \cE\{n\}) = A_{W, \fh}.
\end{align*}
Note that $A^\proj_{W, \fh}$ is the graded algebra controlling $\Pmix(\UGB)$, while $A_{W, \fh}$ agrees with the graded algebra from Definition~\ref{defn:A}. The following result was proved in~\cite{Mak-moment} as a consequence of the fact that Soergel's conjecture implies the isomorphism $\cE_w \cong \IC_w$ for all $w \in W$.
\begin{prop} \label{prop:soergel-koszul}
 If $W$ is finite and $(W, \fh)$ satisfies Soergel's conjecture, then $A^\proj_{W, \fh}$ and $A^\parity_{W, \fh}$ are Koszul, and Koszul dual to each other.
\end{prop}

Given a realization $(\fh, \{\alpha_s^\vee\}, \{\alpha_s\})$, we have the \emph{dual realization} $(\fh^*, \{\alpha_s\}, \{\alpha_s^\vee\})$. We say that $\fh$ is \emph{self-dual} if it is isomorphic to $\fh^*$ as a realization in the obvious sense. Assume that $\fh^*$ is also reflection faithful. Let $R^\vee = \Sym_\bk(\fh)$, graded with $\deg \fh = 2$. Again motivated by geometry, we view the realization $(W, \fh^*)$ as arising from the triple $(G^\vee, B^\vee, T^\vee)$ Langlands dual to $(G, B, T)$. We again use the geometric notation of \S\ref{ss:sbim-notation} for Soergel-theoretic notions asociated to $(W,\fh^*)$. For instance,
\[
 \Parity(\BGUvee) \subset R^\vee\lh\gmod
\]
denotes the category of left Soergel modules associated to $(W, \fh^*)$, with indecomposable objects $\cE_w^\vee$ ``generated'' by the endofunctors
\[
 \theta_s := B_s^\vee \otimes_{R^\vee} (-)\colon R^\vee\lh\gmod \to R^\vee\lh\gmod, \quad B_s^\vee := (R^\vee) \otimes_{(R^\vee)^s} R^\vee \{1\}, \quad s \in S
\]
from $\cE_e^\vee$. Repeating the constructions of \cite{Mak-moment} but with left instead of and right Soergel modules, we obtain in particular the categories
\[
 \Tmix(\BGUvee) \subset \Pmix(\BGUvee) \subset \Dmix(\BGUvee)
\]
and may speak of the objects $\Delta_w^\vee$, $\nabla_w^\vee$, $\cT_w^\vee$ for $w \in W$.

\subsection{Statements}
\label{ss:main-results-statements}
Let $(W, S)$ be a finite Coxeter system, $\bk$ a field of characteristic not equal to 2, and $\fh$ a realization of $(W, S)$ over $\bk$. Assume that both $\fh$ and $\fh^*$ are reflection faithful, so that all categories of the preceding subsections are defined.

Our main result, to be proved in~\S\ref{ss:main-result-proof}, is a Soergel-theoretic analogue of modular Koszul duality~\cite[Theorem~5.4]{AR-II}.
\begin{thm} \label{thm:mkd}
 There exists a triangulated equivalence
 \[
  \kappa\colon \Dmix(\UGB, \bk) \simto \Dmix(\BGUvee, \bk)
 \]
 satisfying $\kappa \circ [1] \cong [1] \circ \kappa$, $\kappa \circ \la1\ra \cong \{1\} \circ \kappa$, $\kappa \circ \{1\} \cong \la1\ra \circ \kappa$, and
 \[
  \kappa(\Delta_w) \cong \Delta^\vee_w, \quad \kappa(\nabla_w) \cong \nabla^\vee_w, \quad \kappa(\cT_w) \cong \cE^\vee_w, \quad \kappa(\cE_w) \cong \cT^\vee_w.
 \]
\end{thm}

We note the following immediate consequences. First, we obtain the following equivalence by composing $\kappa$ with the Ringel self-duality of $\Dmix(\UGB, \bk)$ (proved in~\cite{Mak-moment} by imitating~\cite[Proposition~4.11]{AR-II}).
\begin{cor} \label{cor:koszul-duality}
 There exists a triangulated equivalence
 \[
  \kappa'\colon \Dmix(\UGB, \bk) \simto \Dmix(\BGUvee, \bk)
 \]
 satisfying $\kappa' \circ [1] \cong [1] \circ \kappa'$, $\kappa' \circ \la1\ra \cong \{1\} \circ \kappa'$, $\kappa' \circ \{1\} \cong \la1\ra \circ \kappa'$, and $\kappa'(\cP_{ww_0}) \cong \cE_w$.
\end{cor}
From this equivalence and Proposition~\ref{prop:soergel-koszul}, we deduce the following statement about graded algebras.
\begin{thm} \label{thm:mkd-algebra}
 Suppose that $W$ is finite and $(W, \fh)$ satisfies Soergel's conjecture. Then $A^\proj_{W, \fh}$ is Koszul, and $E(A^\proj_{W, \fh}) \cong A^\proj_{W, \fh^*}$. In particular, if $\fh$ is self-dual, then $A^\proj_{W, \fh}$ is Koszul self-dual.
\end{thm}
The geometric representation $\fh_\geom$ over $\mathbb{R}$ is self-dual. Moreover, Soergel's conjecture for $\fh_\geom$ is a theorem for arbitrary Coxeter systems due to Elias--Williamson~\cite{EW-hodge}. As a result, we obtain a uniform, purely algebraic proof of~Theorem~\ref{thm:ksd-geom}, the Koszul self-duality of $A_{W, \fh_\geom}$ for all finite $W$.

\subsection{Structure of the proof}
\label{ss:proof-structure}
Our proof of Theorem~\ref{thm:mkd} follows an established pattern; we imitate in particular the proof of the ``self-duality'' in the work of Bezrukavnikov--Yun \cite{BY}. In short, the goal is to invent functors $\xi_s$ for $s \in S$ and $\bV$ as in the diagram
\[
 \begin{tikzcd}
  \Dmix(\UGB) \ar[loop, dashed, out=120, in=60, distance=15, "\xi_s"] \ar[r, dashed, "\bV"] & R^\vee\lh\gmod \ar[loop, out=120, in=60, distance=15, "\theta_s"]
 \end{tikzcd}
\]
satisfying the following properties: (1) $\{\xi_s\}_{s \in S}$ ``generate'' $\Tmix(\UGB)$ from the smallest tilting object $\cT_e = \delta$ in the same way that the endofunctors $\{\theta_s\}_{s \in S}$ ``generate'' $\Parity(\BGUvee)$ from the smallest parity object $\cE_e^\vee = \bk$; (2) $\bV(\cT_e) \cong \cE_e^\vee$; (3) $\bV \circ \xi_s \cong \theta_s \circ \bV$ for all $s \in S$. The monodromy action in $\Dmix(\UGB)$, constructed in~\S\ref{sec:monodromy}, will play a key role in defining both $\xi_s$ and $\bV$.

\section{Preliminaries}
\label{sec:preliminaries}
In this section, let $\bk$ be a commutative ring. By ``grading,'' we always mean a $\Z$-grading.

\subsection{Graded modules and graded categories}
\label{ss:graded-categories}
A \emph{graded $\bk$-linear additive category} will mean for us a pair $(\sC, \{1\})$ consisting of a $\bk$-linear additive category $\sC$ and an autoequivalence $\{1\}$, called \emph{(grading) shift}. For $M, N \in \sC$, define the \emph{graded Hom}
\[
 \HOM_\sC(M, N) := \bigoplus_{n \in \Z} \Hom_\sC(M, N\{n\}),
\]
a graded $\bk$-module. For $L, M, N \in \sC$, the obvious induced composition
\[
 (-) \circ (-)\colon \HOM_\sC(M, N) \times \HOM_\sC(L, M) \to \HOM_\sC(L, N),
\]
is graded $\bk$-bilinear. The notation $f\colon M \to N$ will be reserved for an actual morphism of $\sC$. Thus $f\colon M \to N\{n\}$ denotes an element $f \in \HOM_\sC(M, N)$ of degree $n$.

Let $A$ be a graded $\bk$-algebra. We call a graded $\bk$-linear additive category $(\sC, \{1\})$ \emph{graded $A$-linear} if its graded Homs are equipped with the structure of a graded $A$-module in such a way that composition is graded $A$-bilinear.

\begin{ex}
 Let $A\lh\gmod$ denote the category of graded $A$-modules and graded (degree 0) $A$-module homomorphisms. For a graded $A$-module $M = \bigoplus_{i \in \Z} M_i$, define its grading shift $M\{n\}$ by $M\{n\}_i = M_{i+n}$. Then $(A\lh\gmod, \{1\})$ is a graded $\bk$-linear additive category. If $A$ is moreover commutative, then $(A\lh\gmod, \{1\})$ is graded $A$-linear.
\end{ex}

An additive functor between graded $\bk$-linear (resp.~graded $A$-linear) categories is called \emph{graded $\bk$-linear} (resp.~\emph{graded $A$-linear}) if it intertwines the shifts on the nose and the induced maps of graded Hom are graded $\bk$-linear (resp.~$A$-linear).

\subsection{Further background on Soergel (bi)modules and mixed derived category}
\label{ss:further-background}
Let $(W, \fh)$ be a reflection faithful realization.

\subsubsection{A little Soergel diagrammatics} \label{sss:soergel-diagrammatics}
Given $M, N \in R\lh\gmod\lh R$, we write\\$\Hom(M, N)$ for the space of degree 0 $R$-bimodule homomorphisms, and
\[
 \HOM(M, N) := \bigoplus_{ n \in \Z} \Hom(M, N\{n\}) \in R\lh\gmod\lh R
\]
for the graded Hom.

For each $s \in S$, define the following graded $R$-bimodule homomorphisms:
\begin{align*}
 \BsR &\colon B_s \to R\{1\}\colon f \otimes g \mapsto fg  \\
 \RBs &\colon R \to B_s\{1\}\colon f \mapsto f\left(\frac{\alpha_s}{2} \otimes 1 + 1 \otimes \frac{\alpha_s}{2}\right) \\
 \BsRBs := \RBs &\circ \BsR\colon B_s \to B_s\{2\}, \qquad \BsBs := \id_{B_s}\colon B_s \to B_s
\end{align*}
Each diagram is ``$s$-colored,'' but we will always have a fixed $s$ in mind.

These diagrams are borrowed from~\cite{EKh, Eli-dihedral, EW-soergel-calculus}, but we treat them as symbols like any other (rather than embedded graphs up to isotopy). However, the topology of the diagram reminds us whether the $R$-action can be moved from right to left: we have
\begin{gather}
 \label{eq:soergel-diagrammatics-isotopy}
  \BsR f = f \BsR,\quad \RBs f = f \RBs,\quad \BsRBs f = f \BsRBs \quad \text{ for all } f \in R, \\
 \label{eq:polynomial-forcing}
  \BsBs \lambda = s(\lambda) \BsBs + \BsRBs \langle \alpha_s^\vee , \lambda \rangle \in \Hom(B_s, B_s\{2\}) \quad \text{ for all } \lambda \in \fh^*.
\end{gather}

\subsubsection{Soergel Hom formula and equivariant formality} \label{sss:hom-formula-equivariant-formality}
On the Hecke algebra $\cH_W$ (see~\S\ref{ss:coxeter}), define the $\Z[v, v^{-1}]$-bilinear pairing
\[
 \la -, - \ra\colon \cH_W \times \cH_W \to \Z[v, v^{-1}]
\]
determined by $\la H_x, H_y \ra = \delta_{xy}$. The \emph{graded dimension} of a graded vector space $V = \bigoplus_{n \in \Z} V_n$ is
\[
 \gdim V := \sum_{n \in \Z} (\dim V_n)v^n \in \Z[v, v^{-1}].
\]

The following ``equivariant formality'' statement will play an important role throughout this paper.
\begin{prop} \label{prop:sbim-equivariant-formality}
 For any $M, N \in \Parity(\BGB)$, the graded Hom $\HOM(M, N)$ is graded free as a left $R$-module. If moreover $W$ is finite, then the natural map
 \[
  \bk \otimes_R \HOM_{R\lh\gmod\lh R}(M, N) \to \HOM_{\gmod\lh R}(\bk \otimes_R M, \bk \otimes_R N)
 \]
 induced by the functor $\bk \otimes_R (-)\colon R\lh\gmod\lh R \to \gmod\lh R$ is an isomorphism.
\end{prop}
\begin{proof}
 The first statement is part of~\cite[Theorem~5.15]{Soe07}. Soergel originally proved the second statement for the geometric representation of finite Weyl groups \cite[Theorem~2]{Soe92}, and more recently for reflection faithful realizations (see \cite[Proposition~1.13]{Ric}).
\end{proof}
We will also need the Soergel Hom formula~\cite[Theorem~5.15]{Soe07}: given expressions $\ux, \uy$, we have
\[
 \gdim (\bk \otimes_R \HOM(B_\ux, B_\uy)) = \la \uH_\ux, \uH_\uy \ra.
\]
For finite $W$, we deduce by Proposition~\ref{prop:sbim-equivariant-formality} that
\begin{equation} \label{eq:Soergel-module-hom-formula}
 \gdim \HOM_{\gmod\lh R}(\bk \otimes B_\ux, \bk \otimes B_\uy) = \la \uH_\ux, \uH_\uy \ra.
\end{equation}

\subsubsection{Singular Soergel bimodules} \label{sss:ssbim}
The entire story in~\S\ref{ss:mixed-modular} generalizes to singular Soergel theory, or at least to the subregular case. Namely, for every $s \in S$, there are categories
\begin{align*}
 \Parity(\BGPs) &= \oSBim^s(W, \fh) \subset R\lh\gmod\lh R^s, \\
 \Parity(\UGPs) &= \SBim^s(W, \fh) \subset \gmod\lh R^s,
\end{align*}
of singular Soergel (bi)modules~\cite{Wil-ssbim}, with indecomposable objects indexed by the coset space $W/\{ e, s\}$. The geometric notation here is motivated by the fact that, when $(W, \fh)$ arises from a complex reductive group $G$, these categories are related to parity complexes on minimal partial flag varieties $G/P^s$ in the same way as in~\S\ref{ss:relation-to-parity}.

As in the regular case, one can use moment graphs to define the categories
\[
 \Pmix(\UGPs) \subset \Dmix(\UGPs), \quad \Pmix(\BGPs) \subset \Dmix(\BGPs).
\]
Note that $\Dmix(\BGPs)$ is a left module category for $\Dmix(\BGB)$ via $\ast$. There are also exact functors
\[
 \pi^{s*}\colon \Dmix(\BGPs) \to \Dmix(\BGB), \quad \pi^s_*\colon \Dmix(\BGB) \to \Dmix(\BGPs),
\]
and their constructible versions. We noted in~\cite{Mak-moment} that these constructions satisfy various properties one expects from geometry, with mostly the same proofs as in~\cite{AR-II}.

\subsection{Sign convention in homological algebra}
\label{ss:sign-convention}
This technical subsection can safely be skipped on a first reading.

Let $\cA$ be an additive category. Denote by $\Chb\cA$ (resp.~$\Kb\cA$) the category of bounded complexes in $\cA$ (resp.~bounded homotopy category). The usual convention in homological algebra defines a shift functor $\Sigma_\ell$ (usually denoted by $[1]$) introducing a sign in the differential, then defines a triangulated structure on the category with shift $(\Kb\cA, \Sigma_\ell)$.

For certain computations in~\S\ref{sec:monodromy}--\ref{sec:wall-crossing}, it will be more convenient to use a different shift $\Sigma_r$ introducing no sign; the cone will also receive a different sign. Proposition~\ref{prop:shift-triangulated-structure-sign-convention} below assures that this is an inessential choice of convention.

We first recall the usual triangulated structure on $(\Kb\cA, \Sigma_\ell)$, being careful to note the dependence on $\Sigma_\ell$. The shift $\Sigma_\ell$ on $\Chb\cA$ and $\Kb\cA$ is defined by
\[
 (\Sigma_\ell A)^i = A^{i+1}, \quad d_{\Sigma_\ell A}^i = -d_A^{i+1}.
\]
Given a map of complexes $f\colon A \to B$, one defines the \emph{left cone} $C_\ell(f)$ to be the complex
\[
 C_\ell(f)^i = A^{i+1} \oplus B^i, \quad d_{C_\ell(f)}^i = \begin{bmatrix} -d_A^{i+1} & \\ f^{i+1} & d_B^i \end{bmatrix}.
\]
One also associates to $f$ the \emph{left standard triangle}
\[
 S_\ell(f)\colon A \xrightarrow{f} B \xrightarrow{\alpha_\ell(f)} C_\ell(f) \xrightarrow{\beta_\ell(f)} \Sigma_\ell A
\]
in $(\Chb\cA, \Sigma_\ell)$, where $\alpha_\ell(f)$ and $\beta_\ell(f)$ are inclusion and projection. A triangle in $(\Kb\cA, \Sigma_\ell)$ is \emph{left distinguished} if it is isomorphic to the image of some left standard triangle. One then shows that the collection of left distinguished triangles satisfies the axioms of distinguished triangles, hence defines a triangulated structure on $(\Kb\cA, \Sigma_\ell)$.

Define a new shift $\Sigma_r$ on $\Chb\cA$ and $\Kb\cA$ by
\[
 (\Sigma_r A)^i = A^{i+1}, \quad d_{\Sigma_r A}^i = d_A^{i+1}.
\]
For a map of complexes $f\colon A \to B$, define the \emph{right cone} $C_r(f)$ by
\[
 C_r(f)^i = A^{i+1} \oplus B^i, \quad d_{C_r(f)}^i = \begin{bmatrix} d_A^{i+1} & \\ (-1)^i f^{i+1} & d_B^i \end{bmatrix},
\]
and the \emph{right standard triangle}
\[
 S_r(f)\colon A \xrightarrow{f} B \xrightarrow{\alpha_r(f)} C_r(f) \xrightarrow{\beta_r(f)} \Sigma_r A,
\]
where $\alpha_r(f)$ and $\beta_r(f)$ are again inclusion and projection (involving no sign). A triangle in $(\Kb\cA, \Sigma_r)$ is called \emph{right distinguished} if it is isomorphic to the image of some right standard triangle.

\begin{prop} \label{prop:shift-triangulated-structure-sign-convention}
 The collection of right distinguished triangles defines a triangulated structure on $(\Kb\cA, \Sigma_r)$. Moreover, there is a natural isomorphism $\eta\colon \Sigma_\ell \to \Sigma_r$ such that the pair $(\id_{\Kb\cA}, \eta)$ defines a triangulated equivalence $(\Kb\cA, \Sigma_\ell) \cong (\Kb\cA, \Sigma_r)$.
\end{prop}
\begin{proof}
 For any complex $A$, there is an isomorphism of complexes $\eta_A\colon \Sigma_\ell A \to \Sigma_r A$ defined by
 \[
  \eta_A^i = (-1)^i = (-1)^i\id_{A^{i+1}}.
 \]
 This defines a natural isomorphism $\eta\colon \Sigma_\ell \to \Sigma_r$. For any map of complexes $f\colon A \to B$, there is an isomorphism of complexes $\gamma_f\colon C_\ell(f) \to C_r(f)$ defined by
 \[
  \gamma_f^i = \begin{bmatrix} (-1)^i & \\ & 1 \end{bmatrix} = \begin{bmatrix} (-1)^i \id_{A^{i+1}} & \\ & \id_{B^i} \end{bmatrix}.
 \]
 We have the following commutative diagram in $\Chb\cA$:
 \[
  \begin{tikzcd}
   A \ar[r, "f"] \ar[d, equal] & B \ar[r, "\alpha_\ell(f)"] \ar[d, equal] & C_\ell(f) \ar[r, "\beta_\ell(f)"] \ar[d, "\gamma_f"] & \Sigma_\ell A \ar[d, "\eta_A"] \\
   A \ar[r, "f"] & B \ar[r, "\alpha_r(f)"] & C_r(f) \ar[r, "\beta_r(f)"] & \Sigma_r A
  \end{tikzcd}
 \]
 It follows that the isomorphism $(\id_{\Kb\cA}, \eta)\colon (\Kb\cA, \Sigma_\ell) \to (\Kb\cA, \Sigma_r)$ of categories with shift identifies the collection of left distinguished triangles with the collection of right distinguished triangles. The result now follows since the former defines a triangulated structure on $(\Kb\cA, \Sigma_\ell)$.
\end{proof}

\begin{rmk} \label{rmk:sign-convention}
 Let $\cA = \bk\lh\ngmod$, the category of $\bk$-modules. In this case, the shift $\Sigma_\ell$ (resp.~$\Sigma_r$) can be identified with the endofunctor that tensors on the left (resp.~right) with the complex $\bk[1]$ concentrated in degree $-1$. (The differential of a tensor product of complexes is determined by the graded Leibniz rule, with the differential viewed as acting on the left.) This explains the left/right terminology.
\end{rmk}

We use the ``right'' sign convention throughout this paper. We write $A[1]$ for $\Sigma_r(A)$ and drop the subscript $r$ and the adjective ``right.''

\section{Monodromy action}
\label{sec:monodromy}
Let $(W, \fh)$ be a reflection faithful realization (as always, over a field $\bk$ of characteristic not equal to 2). In this section, we define the monodromy action in $\Dmix(\UGB)$.

\subsection{Idea}
\label{ss:monodromy-idea}
To motivate the homological algebra that follows, we first explain the rough idea of the construction.

Let us briefly recall the monodromy action in geometry. Let $G, B, T, U$ be as before, but now defined over a finite field $\mathbb{F}_q$. One can then consider $\Dbm(\BGB)$ (resp.~$\Dbm(\UGB)$), the derived category of $B$-equivariant (resp.~$B$-constructible) mixed $\Qlbar$-complexes on $G/B$ as in \cite{BBD}. These categories are equipped with an autoequivalence $\la1\ra$ called Tate twist.\footnote{More precisely, $\la1\ra := (-\frac12)$, where we fix a square root of $q$ in $\Qlbar$ to make sense of the half Tate twist $(\frac12)$.} Let
\[
 \For\colon \Dbm(\BGB) \to \Dbm(\UGB)
\]
be pullback under the natural projection $\pi: \UGB \to \BGB$. Let $X_*(T)$ be the group of cocharacters of $T$. As explained by Bezrukavnikov--Yun \cite[\S A.1]{BY}, the construction of \cite[\S5]{Ver} applied to the $T$-torsor $\pi$ produces a functorial \emph{(log of) monodromy} action of $\fh := \Qlbar \otimes_\Z X_*(T)$ on any $\cF \in \Dbm(\UGB)$: for any $X \in \fh$, there are morphisms $\mu_{\cF, X}\colon \cF \to \cF\la2\ra$ in $\Dbm(\UGB)$, functorial in $\cF$, such that $\cF$ admits a $B$-equivariant lift if and only if it has trivial monodromy:
\begin{equation} \label{eq:ladic-monodromy}
 \cF \text{ lies in the essential image of } \For \iff \mu_{\cF, X} = 0 \text{ for all } X \in \fh.
\end{equation}

Now, return to our Soergel-theoretic setup: let $(W, \fh)$ be a reflection faithful realization. Our goal is to produce an analogous action in $\Dmix(\UGB)$. An object of $\Dmix(\UGB)$ is a complex $(\cF, d_\cF)$ in $\Parity(\UGB)$:
\[
 \cdots \to \cF^i \xrightarrow{d_\cF^i} \cF^{i+1} \xrightarrow{d_\cF^{i+1}} \cF^{i+2} \to \cdots.
\]
Since $\For\colon \Parity(\BGB) \to \Parity(\UGB)$ is essentially surjective and full, we may (arbitrarily) lift each term and each differential to obtain a ``pre-complex'' $(\widetilde{\cF}, \widetilde{d_\cF})$ in $\Parity(\BGB)$:
\[
 \cdots \to \widetilde{\cF^i} \xrightarrow{\widetilde{d_\cF^i}} \widetilde{\cF^{i+1}} \xrightarrow{\widetilde{d_\cF^{i+1}}} \widetilde{\cF^{i+2}} \to \cdots.
\]
Note that $\widetilde{d_\cF^{i+1}} \circ \widetilde{d_\cF^i}$ may not be $0$, but by Proposition~\ref{prop:sbim-equivariant-formality} it lies in $\fm\HOM(\widetilde{\cF^i}, \widetilde{\cF^{i+2}})$, where $\fm$ denotes the augmentation ideal of $R$. We say that $\widetilde{\cF}$ is a ``pseudo complex'' in $\Parity(\BGB)$. Thus, any $\cF$ admits a pseudo complex lift $\widetilde{\cF}$, and
\begin{equation} \label{eq:new-monodromy}
 \widetilde{\cF} \in \Dmix(\BGB) \iff \widetilde{d_\cF^{i+1}} \circ \widetilde{d_\cF^i} = 0 \text{ for all } i \in \Z.
\end{equation}
Comparing~\eqref{eq:ladic-monodromy} with~\eqref{eq:new-monodromy} suggests that monodromy in $\Dmix(\UGB)$ should measure the failure of $\widetilde{d_\cF} \circ \widetilde{d_\cF}$ to vanish. We illustrate this in an example.
\begin{ex}
 Let $s \in S$. We noted in~\cite{Mak-moment} that the indecomposable tilting object $\cT_s \in \Dmix(\UGB)$ is the image of the three-term complex
 \[
  \cT_s = (\overline{R}\{-1\} \xrightarrow{\overline{\RBs}} \overline{B_s} \xrightarrow{\overline{\BsR}} \overline{R}\{1\}),
 \]
 where $(\cT_s)^0 = \overline{B_s}$ and $\overline{(-)} = \bk \otimes_R (-)$. Since $\BsR \circ \RBs = \alpha_s \id_R$ lies in $\fm\END(R)$, this really is a complex in $\Parity(\UGB)$.

 This complex does not admit a lift to any complex in $\Parity(\BGB)$. To compute the monodromy action, instead lift it to the pseudo complex
 \[
  \widetilde{\cT_s} = (R\{-1\} \xrightarrow{\RBs} B_s \xrightarrow{\BsR} R\{1\}),
 \]
 and consider the ``morphism of pseudo complexes'' $\widetilde{d_{\cT_s}} \circ \widetilde{d_{\cT_s}}\colon \widetilde{\cT_s} \to \widetilde{\cT_s}[2]$:
 \[
  \begin{tikzcd}
   R\{-1\} \ar[r] & B_s \ar[r] & R\{1\} \\
   && R\{-1\} \ar[u, "\alpha_s"] \ar[r] & B_s \ar[r] & R\{1\}.
  \end{tikzcd}
 \]
 Here, the rows depict $\widetilde{\cT_s}[2]$ and $\widetilde{\cT_s}$, and the vertical arrow is the component $(\widetilde{\cT_s})^0 \to (\widetilde{\cT_s}[2])^0$ of $\widetilde{d_{\cT_s}} \circ \widetilde{d_{\cT_s}}$. From this, there is a natural way to cook up a morphism $\cT_s \to \cT_s\la2\ra = \cT_s[2]\{-2\}$: ``dividing through'' by $\alpha_s$ decreases the internal (bimodule) degree by 2, and we get
 \[
  \begin{tikzcd}
   \overline{R}\{-3\} \ar[r] & \overline{B_s}\{-2\} \ar[r] & \overline{R}\{-1\} \\
   && \overline{R}\{-1\} \ar[u, "\id"] \ar[r] & \overline{B_s} \ar[r] & \overline{R}\{1\}.
  \end{tikzcd}
 \]
 This will be the monodromy morphism $\mu_{\cT_s, X}\colon \cT_s \to \cT_s\la2\ra$, where $X \in \fh$ is any element satisfying $\la X, \alpha_s \ra = 1$.
\end{ex}

\subsection{Statement}
\label{ss:monodromy-statement}
The following ad hoc definition will be useful for us.
\begin{defn} \label{defn:monodromic-triple}
 An \emph{$\fh$-monodromic triple} $(\sC_T, \sC_{(T)}, \For)$ consists of a graded $R$-linear category $(\sC_T, \{1\})$, a graded $\bk$-linear category $(\sC_{(T)}, \{1\})$, and a graded $\bk$-linear functor $\For\colon \sC_T \to \sC_{(T)}$, satisfying the following ``equivariant formality'' properties: for $\cF, \cG \in \sC_T$,
 \begin{enumerate}
  \item[\bf(EF1)] $\HOM_{\sC_T}(\cF, \cG)$ is a graded free $R$-module;
  \item[\bf(EF2)] the natural map
  \[
   \bk \otimes_R \HOM_{\sC_T}(\cF, \cG) \to \HOM_{\sC_{(T)}}(\For \cF, \For \cG)
  \]
  is an isomorphism.
 \end{enumerate}
\end{defn}

The bounded homotopy category $\Kb\sC_{(T)}$ has an induced shift $\{1\}$ and a new cohomological shift $[1]$. Define the \emph{Tate twist} $\la1\ra := [1]\{-1\}$. For a $\bk$-linear category with shift $(\mathsf{D}, \Sigma)$, its \emph{graded center} $Z(\mathsf{D}, \Sigma)$ is the graded $\bk$-algebra with degree $n$ part\footnote{Our results only involve the even degree elements of the graded center, so the sign $(-1)^n$ disappears.}
\[
 Z(\mathsf{D}, \Sigma)_n = \{ \alpha\colon \id_\mathsf{D} \to \Sigma^n \mid \alpha_\Sigma = (-1)^n\Sigma \alpha \}.
\]

The following is the main result of this section. Recall that $R^\vee = \Sym_\bk(\fh)$, graded with $\deg \fh = 2$. Let $\fm^\vee$ be its augmentation ideal.
\begin{prop} \label{prop:monodromy}
 Let $(\sC_T, \sC_{(T)}, \For)$ be an $\fh$-monodromic triple. There exists a graded $\bk$-algebra map
 \[
  \mu\colon R^\vee \to Z(\Kb\sC_{(T)}, \la1\ra)\colon f \mapsto \mu_{-, f},
 \]
 called the \emph{monodromy action}, with the following property. For any $\cF \in \Kb\sC_{(T)}$, denote by
 \[
  \mu_\cF\colon R^\vee \to \bigoplus_{n \in \Z} \Hom(\cF, \cF\la n \ra)\colon f \mapsto \mu_{\cF, f}
 \]
 the induced $\bk$-algebra map. If $\cF$ lies in the essential image of the induced functor $\For: \Kb\sC_T \to \Kb\sC_{(T)}$, then $\mu_\cF(\fm^\vee) = \{0\}$, or equivalently, $\mu_\cF(\fh) = \{0\}$.
\end{prop}
The construction of $\mu$ will occupy the rest of this section.

\begin{rmk}
 The intuition for Definition~\ref{defn:monodromic-triple} is that $\fh = \bk \otimes_\Z X_*(T)$ for an algebraic torus $T$, $\sC_T$ (resp.~$\sC_{(T)}$) consists of $T$-equivariant (resp.~$T$-monodromic) $\bk$-sheaves on a $T$-space, and $\For$ forgets the $T$-equivariance. Proposition~\ref{prop:monodromy} should then be thought of as constructing a functorial $T$-monodromy action.
\end{rmk}
\begin{rmk}
 As pointed out to the author by G.~Dhillon, our monodromy action is a special case of the ``cohomology operators'' of \cite{Gul, Eis, AS}. Indeed, the nonstandard homological algebra introduced below is not needed to define our monodromy action. However, we will use the entire setup below to define the wall-crossing functors in \S\ref{sec:wall-crossing}.
\end{rmk}

\subsection{Categories of pseudo complexes}
Let $A$ be a graded $\bk$-algebra, and let $(\sC, \{1\})$ be a graded $A$-linear category. A \emph{graded object} in $\sC$ is a sequence $\cF = (\cF^i)_{i \in \Z}$ of objects in $\sC$. Given graded objects $\cF$ and $\cG$ in $\sC$, define $\cHOM_\sC(\cF, \cG)$, a graded object in $A\lh\gmod$, by
\[
 \cHOM^n_\sC(\cF, \cG) = \prod_{i \in \Z} \HOM_\sC(\cF^i, \cG^{i+n}).
\]
As a $\bk$-module, $\cHOM_\sC(\cF, \cG)$ is bigraded: degree $m$ elements of the graded $A$-module $\cHOM^n_\sC(\cF, \cG)$ are given bidegree $(n, m)$. For graded objects $\cF, \cG, \cH$ in $\sC$, the obvious induced composition
\[
 (-) \circ (-)\colon \cHOM_\sC(\cG, \cH) \times \cHOM_\sC(\cF, \cG) \to \cHOM_\sC(\cF, \cH)
\]
is bigraded and $A$-bilinear.

A (bounded) \emph{pre-complex} $(\cF, d_\cF)$ in $\sC$ consists of a graded object $\cF = (\cF^i)$ in $\sC$, with $\cF^i = 0$ for all but finitely many $i$, together with a degree $(1, 0)$ element $d_\cF \in \cHOM_\sC(\cF, \cG)$, called the \emph{pre-differential}. In components, this consists of morphisms $d_\cF^i\colon \cF^i \to \cF^{i+1}$ in $\sC$.
 
Given pre-complexes $\cF$ and $\cG$ in $\sC$, we make $\cHOM_\sC(\cF, \cG)$ into a pre-complex in $A\lh\gmod$ using the pre-differential
\[
 df = d_\cG \circ f - (-1)^n f \circ d_\cF \quad \mbox{ for } f \in \cHOM^n_\sC(\cF, \cG).
\]
This is the \emph{graded Hom pre-complex} $(\cHOM_\sC(\cF, \cG), d)$. Restricting to elements whose second degree in the bidegree is 0, we get the (non-graded) \emph{Hom pre-complex} \\$(\cHom_\sC(\cF, \cG), d)$, a pre-complex in $\bk\lh\ngmod$.

Now let $(\sC_T, \sC_{(T)}, \For)$ be an $\fh$-monodromic triple. Henceforth, we will drop the subscript $\sC_T$ from the various Homs.
\begin{defn}
 A \emph{pseudo complex} $(\cF, d_\cF)$ (in $\sC_T$) is a pre-complex in $\sC_T$ satisfying
 \begin{equation} \label{eq:pseudo-complex}
  d_\cF \circ d_\cF \in \fm\cHOM^2(\cF, \cF).
 \end{equation}
 We call $d_\cF$ a \emph{pseudo differential}. In components, this consists of morphisms \\$d_\cF^i\colon \cF^i \to \cF^{i+1}$ in $\sC_T$ satisfying $d_\cF^{i+1} \circ d_\cF^i \in \fm\HOM(\cF^i, \cF^{i+2})$.
\end{defn}

Let $\cF, \cG$ be pseudo complexes. Then the graded Hom pre-complex \\$(\cHOM(\cF, \cG), d)$ has the property that for any $n \in \Z$, $d^{n+1} \circ d^n$ maps $\cHOM^n(\cF, \cG)$ into $\fm\cHOM^{n+2}(\cF, \cG)$. Applying $\bk \otimes_R (-)$ termwise yields an honest complex $(\z\!\cHOM(\cF, \cG), \z d)$ in $\bk\lh\gmod$. Thus we have the following commutative diagram, where the vertical arrows are the natural quotient maps:
\[
 \begin{tikzcd}
  \cdots \ar[r] & \cHOM^{-1}(\cF, \cG) \ar[r, "d^{-1}"] \ar[d, twoheadrightarrow] & \cHOM^0(\cF, \cG) \ar[r, "d^0"] \ar[d, twoheadrightarrow] & \cHOM^1(\cF, \cG) \ar[r] \ar[d, twoheadrightarrow] & \cdots \\
  \cdots \ar[r] & \z\cHOM^{-1}(\cF, \cG) \ar[r, "\z d^{-1}"] & \z\cHOM^0(\cF, \cG) \ar[r, "\z d^0"] & \z\cHOM^1(\cF, \cG) \ar[r] & \cdots \\
 \end{tikzcd}
\]
We now define the following graded $R$-linear categories and functors:
\[
 \PChb\sC_T \xrightarrow{P} \z\PChb\sC_T \xrightarrow{Q} \z\PKb\sC_T.
\]
The objects in each category are pseudo complexes. Grading shift $\{1\}$ is induced from that of $\sC_T$. For pseudo complexes $\cF$ and $\cG$, the graded Homs are defined by
\begin{align*}
 \HOM_{\PChb\sC_T}(\cF, \cG) &= (d^0)^{-1}(\fm\cHOM^1(\cF, \cG)), \\
 \HOM_{\z\PChb\sC_T}(\cF, \cG) &= \ker(\z d^0), \\
 \HOM_{\z\PKb\sC_T}(\cF, \cG) &= \ker(\z d^0)/\im(\z d^{-1}).
\end{align*}
The functors $P, Q$ are the identity map on objects and natural quotient maps on morphisms.

\begin{defn}
 Morphisms in $\PChb\sC_T$ are called \emph{pseudo maps}. Thus, a pseudo map $\varphi\colon \cF \to \cG$ is an element of $\cHom^0(\cF, \cG)$ satisfying
 \begin{equation} \label{eq:pseudo-map}
  d\varphi \in \fm\cHOM^1(\cF, \cG)
 \end{equation}
 In components, this consists of morphisms $\varphi^i\colon \cF^i \to \cG^i$ in $\sC_T$ satisfying
 \[
  d_\cG^i \circ \varphi^i - \varphi^{i+1} \circ d_\cF^{i+1} \in \fm\HOM(\cF^i, \cG^{i+1}).
 \]
\end{defn}

If $(\cF, d_\cF)$ is a pseudo complex in $\sC_T$, then {\bf(EF2)} implies that $(\For\cF, \For d_\cF)$ is a complex in $\sC_{(T)}$. Moreover, for pseudo complexes $\cF$ and $\cG$, the isomorphism of {\bf(EF2)} induces an identification of complexes
\[
 (\z\cHOM(\cF, \cG), \z d) \cong (\cHOM(\For\cF, \For\cG), d).
\]
The latter complex is the graded version of the Hom complex used in ordinary homological algebra to define $\Chb\sC_{(T)}$ and $\Kb\sC_{(T)}$. Hence this identification induces equivalences $R_\Ch\colon \z\PChb\sC_T \to \Chb\sC_{(T)}$ and $R_K\colon \z\PKb\sC_T \to \Kb\sC_{(T)}$.

The situation is summarized in the following diagram:
\[
 \begin{tikzcd}
  \PChb\sC_T \ar[r, "P"] & \z\PChb\sC_T \ar[r, "Q"] \ar[d, "R_\Ch", "\wr"'] & \z\PKb\sC_T \ar[d, "R_K", "\wr"'] \\
                         & \Chb\sC_{(T)} \ar[r, "q"]                        & \Kb\sC_{(T)}
 \end{tikzcd}
\]
Each category is graded $R$-linear with grading shift $\{1\}$, and has an additional autoequivalence, the homological shift $[1]$. Define the \emph{Tate twist} $\langle 1\rangle = [1]\{-1\}$. All functors are graded $R$-linear and commute with $\{1\}, [1], \la1\ra$ on the nose.

\subsection{Triangulated structure on $\z\PKb\sC_T$}
\label{ss:triangulated-structure-on-PKb}
Let $\varphi\colon \cF \to \cG$ be a pseudo map of pseudo complexes. Define the cone pseudo complex $C(\varphi)$ and standard triangle
\[
 \p S(\varphi)\colon \cF \xrightarrow{\varphi} \cG \xrightarrow{\alpha(\varphi)} C(\varphi) \xrightarrow{\beta(\varphi)} \cF[1]
\]
in $\PChb\sC_T$ by the same formula as for complexes.

Since the standard triangles in $\Chb\sC_{(T)}$ are precisely triangles of the form\\$R_\Ch P (\p S(\varphi))$, distinguished triangles in $\Kb\sC_{(T)}$ are triangles that are isomorphic to some $q R_\Ch P (\p S(\varphi))$. Define a triangle $T$ in $\z\PKb\sC_T$ to be distinguished if $T \cong QP(\p S(\varphi))$ for some $\varphi$. This is equivalent to $R_K(T) \cong R_K Q P (S(\varphi)) = q R_\Ch P (S(\varphi))$, i.e.~$R_K(T)$ is isomorphic to a standard triangle in $K^b\sC_{(T)}$. From this description, we deduce the following result.
\begin{lem}
 The definitions above define a triangulated structure on $\z\PKb\sC_T$ making $R_K$ exact.
\end{lem}

\subsection{Construction of monodromy}
Let $M$ be a graded free $R$-module. For any $X \in \fh$, define a graded $\bk$-module map
\[
 \Phi_{M, X}\colon \fm M \to \z M\{-2\}
\]
as follows. If $X = 0$, set $\Phi_{M, X} = 0$. Otherwise, extend $X$ to a basis $\{X = X_1, X_2, \ldots, X_r\}$ of $\fh$, with dual basis $\{X^*_1, \ldots, X^*_r\}$ of $\fh^*$. Any element $m \in \fm M$ can be written as $m = X_1^*m_1 + \cdots + X_r^*m_r$ for some $m_i \in M$. Although the elements $m_i$ are not unique, $M$ being graded free ensures that their classes $\z m_i \in \z M$ are well-defined; set $\Phi_{M, X_i}(m) = \z m_i$. It is easy to see that this does not depend on the choice of basis.

The following result is straightforward.
\begin{lem} \label{lem:monodromy-module}
 Let $M, N$ be graded free $R$-modules.
 \begin{enumerate}
  \item $\Phi_{M, X}$ is linear in $X$.
  \item For any $m \in M$, $\beta \in \fh^*$, and $X \in \fh$, we have
   \[
    \Phi_{M, X}(\beta m) = \langle X, \beta \rangle \z m.
   \]
  \item For any graded $R$-linear map $f\colon M \to N$ and $X \in \fh$, we have
  \[
   \z f \circ \Phi_{M, X} = \Phi_{N, X} \circ f|_{\fm M}.
  \]
 \end{enumerate}
\end{lem}

We apply this lemma to $\cHOM(\cF, \cG)$, which is graded free by {\bf(EF1)}.
\begin{lem} \label{lem:monodromy-cHOM}
 Let $\cF, \cG, \cH$ be pre-complexes in $\sC_T$, and let $X \in \fh$.
 \begin{enumerate}
  \item If $\varphi \in \fm\cHOM(\cF, \cG)$ (so that $d\varphi \in \fm\cHOM(\cF, \cG)$), then
  \[
   \z d(\Phi_{\cHOM(\cF, \cG), X}(\varphi)) = \Phi_{\cHOM(\cF, \cG), X}(d\varphi).
  \]
  \item If $\varphi \in \cHOM^0(\cF, \cG)$ and $\psi \in \fm\cHOM^1(\cG, \cH)$ (so that \\$\psi \circ \varphi \in \fm\cHOM^1(\cF, \cH)$), then
  \[
   \Phi_{\cHOM^1(\cF, \cH)}(\psi \circ \varphi) = \Phi_{\cHOM^1(\cG, \cH)}(\psi) \circ \z\varphi.
  \]
  If $\varphi \in \fm\cHOM^1(\cF, \cG)$ and $\psi \in \cHOM^0(\cG, \cH)$, then
  \[
   \Phi_{\cHOM^1(\cF, \cH)}(\psi \circ \varphi) = \z\psi \circ \Phi_{\cHOM^1(\cF, \cG)}(\varphi).
  \]
 \end{enumerate}
\end{lem}
\begin{proof}
 \begin{enumerate}
  \item Apply Lemma~\ref{lem:monodromy-module}(3) to $d\colon \cHOM(\cF, \cG) \to \cHOM(\cF, \cG)$.
  \item For the first equality, apply Lemma~\ref{lem:monodromy-module}(3) to $- \circ \varphi\colon \cHOM^1(\cG, \cH) \to \cHOM^1(\cF, \cH)$. The second equality is similar. \qedhere
 \end{enumerate}
\end{proof}

Let $(\cF, d_\cF)$ be a pseudo complex, and $\varphi\colon \cF \to \cG$ a pseudo map of pseudo complexes. By \eqref{eq:pseudo-complex} (resp.~\eqref{eq:pseudo-map}), we can define, for any $X \in \fh$,
\[
 \mu_{\cF, X} = \Phi_{\cHom^2(\cF, \cF), X}(d_\cF \circ d_\cF) \qquad (\mbox{resp.~}\nu_{\varphi, X} = \Phi_{\cHom^1(\cF, \cG), X}(d\varphi)).
\]
This is a degree $-2$ element in $\z\cHOM^2(\cF, \cF)$ (resp.~$\z\cHOM^1(\cF, \cG)$) that measures the failure of $\cF$ to be a complex (resp.~the failure of $\varphi$ to be a map of complexes) ``in the $X$ direction.''

\begin{lem} \label{lem:mu-and-nu}
  Let $\varphi\colon \cF \to \cG$ and $\psi\colon \cG \to \cH$ be pseudo maps of pseudo complexes. For all $X \in \fh$, we have
  \begin{align*}
   \z d(\mu_{\cF, X}) &= 0, \\
   \z d(\nu_{\varphi, X}) &= \mu_{\cG, X} \circ \z\varphi - \z\varphi \circ \mu_{\cF, X}, \\
   \nu_{\psi \circ \varphi, X} &= \nu_{\psi, X} \circ \z\varphi + \z\psi \circ \nu_{\varphi, X}.
  \end{align*}
\end{lem}
\begin{proof}
 Each equation is straightforward to prove from the definitions and \\Lemma~\ref{lem:monodromy-cHOM}.
\end{proof}

We view $\mu_{\cF, X}$ as a degree 0 element of $\z\cHOM^0(\cF, \cF\langle 2 \rangle)$ via the natural identification of complexes
\[
 (\z\cHOM^n(\cF, \cF), \z d^n) \cong (\z\cHOM^{n-2}(\cF, \cF\langle 2 \rangle)\{2\}, \z d^{n-2}).
\]
By Lemma~\ref{lem:mu-and-nu}(1), $\mu_{\cF, X}$ defines a morphism $\cF \to \cF\langle 2 \rangle$ in both $\z\PChb\sC_T$ and $\z\PKb\sC_T$. We similarly view $\nu_{\varphi, X}$ as a degree $0$ element of $\z\cHOM^{-1}(\cF, \cG \langle 2 \rangle)$.

\begin{lem} \label{lem:mu-as-graded-center}
 Any $\mu_{-, X}$, $X \in \fh$, is a degree 2 element of $Z(\z\PChb\sC_T, \la1\ra)$.
\end{lem}
\begin{proof}
 It is clear from the construction that $\mu_{-, X}$ commutes with $\la1\ra$. For any pseudo map $\varphi\colon \cF \to \cG$, we must show that
 \[
  \z\varphi\la2\ra \circ \mu_{\cF, X} = \mu_{\cG, X} \circ \z\varphi
 \]
 in $\z\PKb\sC_T$. This follows from Lemma~\ref{lem:mu-and-nu}(2), which translates to
 \[
  \z d(\nu_{\varphi, X}) = \mu_{\cG, X} \circ \z\varphi - \z\varphi\langle 2 \rangle \circ \mu_{\cF, X} \in \z\cHOM^0(\cF, \cG\langle 2 \rangle). \qedhere
 \]
\end{proof}
\begin{proof}[Proof of~Proposition~\ref{prop:monodromy}]
 Since $\mu_{\cF, X}$ is linear in $X$, the map $X \mapsto \mu_{-, X}$ extends uniquely to a graded $\bk$-algebra map
 \[
  \mu\colon R^\vee \to Z(\PKb\sC_T, \langle 1 \rangle)\colon f \mapsto \mu_{-, f}.
 \]
 An equivalence of categories with shift induces an isomorphism of their graded centers. Hence, using $R_K$, we obtain the desired map $\mu$. The last statement is clear from the construction.
\end{proof}

\section{Wall-crossing functors}
\label{sec:wall-crossing}
Let $(W, \fh)$ be a reflection faithful realization, and fix $s \in S$. In this section, we will construct an endofunctor $\xi_s$ of $\Dmix(\UGB)$. These functors will restrict to exact functors on $\Pmix(\UGB)$, where they correspond to the wall-crossing functors on $\cO_0^\gr$.

\subsection{Idea}
\label{ss:wall-crossing-idea}
Return for a moment to the geometric setup of \S\ref{ss:monodromy-idea}. Let $\Dbm(\UGU)$ be the derived category of $U$-equivariant mixed complexes on the enhanced flag variety $G/U$. Bezrukavnikov--Yun~\cite{BY} defined a category $\hDbm(\UGUby)$, monoidal under $U$-convolution $\stackrel{U}{\ast}$, as a certain (``free-monodromic'') completion of a full subcategory of $\Dbm(\UGU)$. This category $\hDbm(\UGUby)$ acts by $\stackrel{U}{\ast}$ on the left of $\Dbm(\UGB)$ and contains the ``free-monodromic'' tilting sheaf $\widetilde{\cT}_s$, giving the endofunctor
\[
 \xi_s = \widetilde{\cT}_s \stackrel{U}{\ast} (-)\colon \Dbm(\UGB) \to \Dbm(\UGB).
\]
This is the functor we wish to imitate in our Soergel-theoretic setting.

Although we do not have an analogue of $\hDbm(\UGUby)$, we can start to guess the definition of $\xi_s$ as follows. First, since proper pushforward along the natural projection $\UGU \to \UGB$ sends $\widetilde{\cT}_s$ to $\cT_s$, the following diagram commutes up to natural isomorphism by proper base change:
\begin{equation} \label{diag:wall-crossing-geometry}
 \begin{tikzcd}
  \Dbm(\UGB) \ar[rd, "\xi_s = \widetilde{\cT}_s \stackrel{U}{\ast} (-)"] \\
  \Dbm(\BGB) \ar[u, "\For"] \ar[r, "\cT_s \stackrel{B}{\ast} (-)"'] & \Dbm(\UGB).
 \end{tikzcd}
\end{equation}

Now, back in our setting, we do have an analogue of the functor $\cT_s \stackrel{B}{\ast} (-)$:
\[
 \cT_s\otimes_R(-)\colon \Dmix(\BGB) \to \Dmix(\UGB).
\]
From here on, we omit $\otimes_R$ from the notation. Explicitly, this functor sends $\cF$ in $\Dmix(\BGB)$ to the complex $\cT_s\cF$ is given by
\[
 (\cT_s\cF)^i = \cF^{i-1}\{1\} \oplus B_s\cF^i \oplus \cF^{i+1}\{-1\}, \quad
 d_{\cT_s\cF}^i = \begin{bmatrix}
                 -d_\cF^{i-1} & \BsR 1_{\cF^i} & \\
                              & \BsBs d_\cF^i  & \RBs 1_{\cF^{i+1}} \\
                              &                & -d_\cF^{i+1}
                \end{bmatrix},
\]
and the morphism $\varphi\colon \cF \to \cG$ to $1_{\cT_S}\varphi\colon \cT_s\cF \to \cT_s\cG$, given by
\[
 (1_{\cT_s}\varphi)^i = \diag(\varphi^{i-1}, \BsBs \varphi^i, \varphi^{i+1})\colon (\cT_s\cF)^i \to (\cT_s\cG)^i.
\]
Here, each matrix entry is viewed modulo $\fm$, i.e.~as morphisms in $\Parity(\UGB)$.

\begin{figure}
 \[
 \scalebox{0.8}{
  \xymatrix@R=0.8em@C=1.3em@M=0.3em{
                 &                                                                       & \ddots \ar[rd]                                                          &                                                                       & \\
                 & \ddots \ar[rd]                                                        &                                                                         & \cF^{i-1}\{1\} \ar[rd]^-{-d_\cF^{i-1}}                                & \\
  \ddots \ar[rd] &                                                                       & B_s\cF^{i-1} \ar[ru]^-{\BsR 1_{\cF^{i-1}}} \ar[rd]^-{\BsBs d_\cF^{i-1}} &                                                                       & \cF^i\{1\} \ar[rd]^-{-d_\cF^i}                                          & \\
                 & \cF^{i-1}\{-1\} \ar[ru]^-{\RBs 1_{\cF^{i-1}}} \ar[rd]_-{-d_\cF^{i-1}} &                                                                         & B_s\cF^i \ar[ru]^-{\BsR 1_{\cF^i}} \ar[rd]^-{\BsBs d_\cF^i}           &                                                                         & \cF^{i+1}\{1\} \ar[rd]^-{-d_\cF^{i+1}} \\
                 &                                                                       & \cF^i\{-1\} \ar[ru]^-{\RBs 1_{\cF^i}} \ar[rd]_-{-d_\cF^i}               &                                                                       & B_s\cF^{i+1} \ar[ru]^-{\BsR 1_{\cF^{i+1}}} \ar[rd]^-{\BsBs d_\cF^{i+1}} &                                                    & \cF^{i+2}\{1\} \ar[rd] & \\
                 &                                                                       &                                                                         & \cF^{i+1}\{-1\} \ar[ru]^-{\RBs 1_{\cF^{i+1}}} \ar[rd]_-{-d_\cF^{i+1}} &                                                                         & B_s\cF^{i+2} \ar[ru]^-{\BsR 1_{\cF^{i+2}}} \ar[rd] &                        & \ddots \\
                 &                                                                       &                                                                         &                                                                       & \cF^{i+2}\{-1\} \ar[ru]^-{\RBs 1_{\cF^{i+2}}} \ar[rd]                   &                                                    & \ddots \\
                 &                                                                       &                                                                         &                                                                       &                                                                         & \ddots \\
  }
  }
 \]
 \caption{The pre-complex $\cT_s\cF$ for $\cF \in \Dmix(\UGB)$} \label{fig:TsF}
\end{figure}
\begin{figure}
 \[
 \scalebox{0.8}{
  \xymatrix@R=0.8em@C=1.3em@M=0.3em{
                 &                                 & \ddots \ar[rd]                                                                 &                                                                                            & \\
                 & \ddots \ar[rd]                  &                                                                                & \cF^{i-1}\{1\} \ar[rd] \ar@{~>}[rddd]                                                      & \\
  \ddots \ar[rd] &                                 & B_s\cF^{i-1} \ar[ru] \ar[rd] \ar@{-->}[rddd]                                   &                                                                                            & \cF^i\{1\} \ar[rd] \ar@{~>}[rddd]                                          &  \\
                 & \cF^{i-1}\{-1\} \ar[ru] \ar[rd] &                                                                                & B_s\cF^i \ar[ru]^{\RBs \mu_{\cF, \alpha_s^\vee}^{i-1}\phantom{MM}} \ar[rd] \ar@{-->}[rddd] &                                                                            & \cF^{i+1}\{1\} \ar[rd] \\
                 &                                 & \cF^i\{-1\} \ar[ru]^{\BsR \mu_{\cF, \alpha_s^\vee}^{i-1}\phantom{MMM}} \ar[rd] &                                                                                            & B_s\cF^{i+1} \ar[ru]^{\RBs \mu_{\cF, \alpha_s^\vee}^i\phantom{MM}} \ar[rd] &                              & \cF^{i+2}\{1\} \ar[rd] & \\
                 &                                 &                                                                                & \cF^{i+1}\{-1\} \ar[ru]^{\BsR \mu_{\cF, \alpha_s^\vee}^i\phantom{MMM}} \ar[rd]             &                                                                            & B_s\cF^{i+2} \ar[ru] \ar[rd] &                        & \ddots \\
                 &                                 &                                                                                &                                                                                            & \cF^{i+2}\{-1\} \ar[ru] \ar[rd]                                            &                              & \ddots \\
                 &                                 &                                                                                &                                                                                            &                                                                            & \ddots \\
  }
  }
 \]
 \caption{The pre-complex $\cT_s\cF$ with correction components} \label{fig:TsF-with-correction}
\end{figure}

By analogy with~\eqref{diag:wall-crossing-geometry}, the desired functor $\xi_s$ should extend this to $\Dmix(\UGB)$, i.e.~to a map of complexes $\varphi\colon\cF \to \cG$ in $\Parity(\UGB)$. There are two difficulties. First, since $d_\cF^i$ and $\varphi^i$ are only defined modulo $\fm$ (on the left!), $\BsBs d_\cF^i$ and $\BsBs \varphi^i$ are not well-defined modulo $\fm$. We solve this by lifting these to actual bimodule maps, i.e.~work with a pseudo map $\varphi\colon \cF \to \cG$ of pseudo complexes. Second, for a pseudo complex $\cF$ (resp.~pseudo map $\varphi\colon \cF \to \cG$), the pre-complex $\cT_s\cF$ (resp.~pre-complex map $1_{\cT_s}\varphi$) is in general not a pseudo complex (resp.~pseudo map); see Figure~\ref{fig:TsF}. Indeed, \eqref{eq:polynomial-forcing} implies that for the component
\[
 \BsBs (d_\cF \circ d_\cF)\colon B_s\cF \to B_s\cF \qquad (\text{resp.} \BsBs d\varphi\colon B_s\cF \to B_s\cG)
\]
of $d_{\cT_s\cF} \circ d_{\cT_s\cF}$ (resp.~$d(1_{\cT_s}\varphi)$) to vanish modulo $\fm$, $d_\cF \circ d_\cF$ (resp.~$d\varphi$) must vanish modulo $\fm^2$.

The following key computation tells us how to proceed. Choose a basis \\$\{X_1, \ldots, X_r\}$ of $\fh$, and write
\[
 d_\cF \circ d_\cF = \sum X_j^* \widetilde{\mu}_{\cF, X_j} \quad\mbox{ for some } \widetilde{\mu}_{\cF, X_j} \in \cHOM^2(\cF, \cF),
\]
so that $\widetilde{\mu}_{\cF, X_j}$ lifts $\mu_{\cF, X_j}$. Then by \eqref{eq:polynomial-forcing} and the linearity of $\mu_{\cF, -}$,
\begin{equation} \label{eq:wall-crossing-key-computation-objects}
 \BsBs (d_\cF \circ d_\cF) = \BsBs \sum X_j^* \tilde{\mu}_{\cF, X_j} = \BsRBs \sum \langle \alpha_s^\vee, X_j^* \rangle \mu_{\cF, X_j} = \BsRBs \mu_{\cF, \alpha_s^\vee}
\end{equation}
in $\z\cHOM^2(\cF, \cF)$. Therefore, we may add correction components $-\BsR \mu_{\cF, \alpha_s^\vee}$ (or $-\RBs \mu_{\cF, \alpha_s^\vee}$) as in Figure~\ref{fig:TsF-with-correction} to turn $\cT_s\cF$ into an actual complex in $\Parity(\UGB)$. (Since $\mu_{\cF, \alpha_s^\vee}$ is only defined modulo $\fm$, we do not get a pseudo complex in \\$\Parity(\BGB)$.) Similarly, the computation
\begin{equation} \label{eq:wall-crossing-key-computation-morphisms}
 \BsBs d\varphi = \BsBs \sum X_j^* \tilde{\nu}_{\varphi, X_j} = \BsRBs \sum \langle \alpha_s^\vee, X_j^* \rangle \nu_{\varphi, X_j} = \BsRBs \nu_{\varphi, \alpha_s^\vee}
\end{equation}
in $\z\cHOM^1(\cF, \cG)$ suggests adding correction components $-\BsR \nu_{\varphi, \alpha_s^\vee}$ (or $-\RBs \nu_{\varphi, \alpha_s^\vee}$) to turn $1_{\cT_s}\varphi$ into an actual map of complexes in $\Parity(\UGB)$.

\subsection{Statement}
\label{ss:statement}
Let $\sC_{s,T} \subset \Parity(\BGB)$ be the full subcategory consisting of $R\{n\}$ and $B_s\{n\}$ for $n \in \Z$. Let $(\sC_T, \sC_{(T)}, \For)$ be an $\fh$-monodromic triple. Suppose that we have a bifunctor
\[
 (-)\ast(-)\colon \sC_{s,T} \times \sC_T \to \sC_T
\]
such that the induced maps on graded Hom are graded, $R$-linear on the left, and $R$-middle-linear. These assumptions ensure that we have an induced exact functor $\cT_s \ast (-)\colon \Kb\sC_T \to \Kb\sC_{(T)}$ as in~\S\ref{ss:wall-crossing-idea}, and that the key computations \eqref{eq:wall-crossing-key-computation-objects} and \eqref{eq:wall-crossing-key-computation-morphisms} still hold.

The following is the main result of this section.
\begin{prop} \label{prop:wall-crossing}
 In the situation above, there exists a functor $\xi_s\colon \Kb\sC_{(T)} \to \Kb\sC_{(T)}$, defined up to natural isomorphism, with the following properties:
 \begin{enumerate}
  \item $\xi_s \circ \{1\} = \{1\} \circ \xi_s$ and $\xi_s \circ [1] = [1] \circ \xi_s$.
  \item $\xi_s$ is exact.
  \item There is a natural isomorphism $\cT_s \ast (-) \cong \xi_s \circ \For$.
  \item Let $\cF \in \Kb\sC_{(T)}$ and $f \in (R^\vee)^s$, homogeneous of degree $d$. Then
   \[
    \xi_s\mu_{\cF, f} = \mu_{\xi_s\cF, f}\colon \xi_s\cF \to \xi_s\cF\la d \ra.
   \]
  \item Let $(\sC'_T, \sC'_{(T)}, \For)$ be another $\fh$-monodromic triple equipped with a bifunctor $(-) \ast (-): \sC_{s,T} \times \sC'_T \to \sC'_T$ as above. Let $\xi'_s: \Kb\sC'_{(T)} \to \Kb\sC'_{(T)}$ be the resulting endofunctor. Let $F\colon \sC_T \to {\sC'}_T$ be a graded $R$-linear functor, inducing a functor $F\colon \Kb\sC_{(T)} \to \Kb\sC'_{(T)}$. If $F$ intertwines $\ast$ and $\ast'$ up to natural isomorphism, then $F \circ \xi_s \cong \xi'_s \circ F$.
 \end{enumerate}
\end{prop}
We construct $\xi_s$ in~\S\ref{ss:wall-crossing-construction}. The proof of Proposition~\ref{prop:wall-crossing} occupies~\S\ref{ss:wall-crossing-proof}.
\begin{rmk}
 We insist on equality, not just natural isomorphism, in (1) to simplify certain computations with $\xi_s$. In particular, the exactness of $\xi_s$ is proved by checking directly that it sends a standard triangle to a distinguished triangle. The sign convention in~\S\ref{ss:sign-convention} is chosen to simplify this computation.
\end{rmk}

\subsection{Construction of $\xi_s$}
\label{ss:wall-crossing-construction}
We will proceed as follows (see diagram \eqref{eq:wall-crossing-construction-outline} below): define a functor $\p\xi_s\colon \PChb\sC_T \to \Chb\sC_{(T)}$ by explicitly modifying $\cT_s \ast (-)$ as in~\S\ref{ss:wall-crossing-idea}; show that $q \circ \p\xi_s$ factors (uniquely) through a functor $\z(\p\xi_s)\colon \z\PKb\sC_T \to \Kb\sC_{(T)}$; finally, choose a quasi-inverse $R_K^{-1}$ commuting with $\{1\}$ and $[1]$ on the nose, and set $\xi_s = \z(\p\xi_s) \circ R_K^{-1}$.
\begin{equation} \label{eq:wall-crossing-construction-outline}
 \begin{tikzcd}
  \PChb\sC_T \ar[r, "P"] \ar[rd, dashed, "\p\xi_s"'] & \z\PChb\sC_T \ar[r, "Q"] \ar[d, "R_\Ch"] & \z\PKb\sC_T \ar[d, shift left=3pt, "R_K"] \ar[d, dashed, shift right=3pt, "\z(\p\xi_s)"'] \\
                                & \Chb\sC_{(T)} \ar[r, "q"] & \Kb\sC_{(T)}
 \end{tikzcd}
\end{equation}

\begin{rmk}
 Recall from~\S\ref{ss:wall-crossing-idea} that we could add the correction components to $\cT_s \ast (-)$ in two ways. Here, we work with one of these ways. Both choices lead to naturally isomorphic $\z(\p\xi_s)$, hence naturally isomorphic $\xi_s$.
\end{rmk}

\subsubsection*{Step 1: Define $\p\xi_s$}
As in~\S\ref{ss:wall-crossing-idea}, we omit $\ast$ from the notation. Define $\p\xi_s$ on objects by
\[
 (\p\xi_s\cF)^i = \cF^{i-1}\{1\} \oplus B_s\cF^i \oplus \cF^{i+1}\{-1\}, \quad
 d_{\p\xi_s\cF}^i =
 \begin{bmatrix}
  -d_\cF^{i-1} & \BsR 1_{\cF^i}                    & \\
               & \BsBs d_\cF^i                     & \RBs 1_{\cF^{i+1}} \\
               & - \BsR \mu_{\cF, \alpha_s^\vee}^i & -d_\cF^{i+1}
 \end{bmatrix},
\]
and on morphisms by
\[
 (\p\xi_s\varphi)^i =
 \begin{bmatrix}
  \varphi^{i-1} &                                      & \\
                & \BsBs \varphi^i                            & \\
                & -\BsR \nu_{\varphi, \alpha_s^\vee}^i & \varphi^{i+1}
 \end{bmatrix},
\] 
where all matrix entries are to be viewed as a morphism in $\sC_{(T)}$ (e.g.~the top left entry of $(\p\xi_s\cF)^i$ is in $\z\HOM(\cF^{i-1}\{1\}, \cF^i\{1\})$. This is illustrated in Figure~\ref{fig:pxi}.

\newsavebox\boxone
\savebox\boxone{%
 $\left[
  \begin{smallmatrix}
  -d_\cG^{i-1} & \BsR 1_{\cG^i}                    & \\
               & \BsBs d_\cG^i                     & \RBs 1_{\cG^{i+1}} \\
               & - \BsR \mu_{\cG, \alpha_s^\vee}^i & -d_\cG^{i+1}
  \end{smallmatrix}
  \right]$
}
\newsavebox\boxtwo
\savebox\boxtwo{%
 $\left[
  \begin{smallmatrix}
   -d_\cF^{i-1} & \BsR 1_{\cF^i} & \\
                & \BsBs d_\cF^i         & \RBs 1_{\cF^{i+1}} \\
                & - \BsR \mu_{\cF, \alpha_s^\vee}^i & -d_\cF^{i+1}
  \end{smallmatrix}
  \right]$
}
\newsavebox\boxthree
\savebox\boxthree{%
 $\left[
  \begin{smallmatrix}
   \varphi^{i-1} &                                      & \\
                 & \BsBs \varphi^i                      & \\
                 & -\BsR \nu_{\varphi, \alpha_s^\vee}^i & \varphi^{i+1}
  \end{smallmatrix}
  \right]$
}
\newsavebox\boxfour
\savebox\boxfour{%
 $\left[
  \begin{smallmatrix}
   \varphi^i &                                          & \\
             & \BsBs \varphi^{i+1}                      & \\
             & -\BsR \nu_{\varphi, \alpha_s^\vee}^{i+1} & \varphi^{i+2}
  \end{smallmatrix}
  \right]$
}

\begin{figure}
\[
 \scalebox{0.9}{
 \begin{tikzcd}[row sep=huge, ampersand replacement=\&]
  \cdots \ar[r]
  \&
  {\begin{matrix} R\cG^{i-1}\{1\} \\ B_s\cG^i \\ R\cG^{i+1}\{-1\}\end{matrix}} \ar[rrrrr, "\usebox\boxone"]
  \&\&\&\&\&
  {\begin{matrix} R\cG^i\{1\} \\ B_s\cG^{i+1} \\ R\cG^{i+2}\{-1\}\end{matrix}} \ar[r]
  \&
  \cdots \\
  \cdots \ar[r]
  \&
  {\begin{matrix} R\cF^{i-1}\{1\} \\ B_s\cF^i \\ R\cF^{i+1}\{-1\}\end{matrix}} \ar[rrrrr, "\usebox\boxtwo"'] \ar[u, "\usebox\boxthree"]
  \&\&\&\&\&
  {\begin{matrix} R\cF^i\{1\} \\ B_s\cF^{i+1} \\ R\cF^{i+2}\{-1\}\end{matrix}} \ar[r] \ar[u, "\usebox\boxfour"']
  \&
  \cdots
 \end{tikzcd}
 }
\]
\caption{Effect of $\p\xi_s$ on $\varphi\colon \cF \to \cG$ (in degrees $i$ and $i+1$)} \label{fig:pxi}
\end{figure}

To check that this is a map of complexes in $\sC_{(T)}$, we compute (now omitting indices)
\begin{multline*}
 d_{\p\xi_s\cG} \circ (\p\xi_s\varphi) =
 \begin{bmatrix}
  -d_\cG & \BsR 1_{\cG}                    & \\
         & \BsBs d_\cG                     & \RBs 1_\cG \\
         & - \BsR \mu_{\cG, \alpha_s^\vee} & -d_\cG
 \end{bmatrix}
 \begin{bmatrix}
  \varphi &                                    & \\
          & \BsBs \varphi                      & \\
          & -\BsR \nu_{\varphi, \alpha_s^\vee} & \varphi
 \end{bmatrix} \\
 =
 \begin{bmatrix}
  -d_\cG \circ \varphi & \BsR \varphi                                                                                     & \\
                       & \BsBs (d_\cG \circ \varphi) - \BsRBs \nu_{\varphi, \alpha_s^\vee}                                & \RBs \varphi \\
                       & -\BsR (\mu_{\cG, \alpha_s^\vee} \circ \varphi) + \BsR (d_\cG \circ \nu_{\varphi, \alpha_s^\vee}) & -d_\cG \circ \varphi
 \end{bmatrix}
\end{multline*}
and
\begin{multline*}
 (\p\xi_s\varphi) \circ d_{\p\xi_s\cF} =
 \begin{bmatrix}
  \varphi &                                    & \\
          & \BsBs \varphi                      & \\
          & -\BsR \nu_{\varphi, \alpha_s^\vee} & \varphi
 \end{bmatrix}
 \begin{bmatrix}
  -d_\cF & \BsR 1_{\cF}                    & \\
         & \BsBs d_\cF                     & \RBs 1_{\cF} \\
         & - \BsR \mu_{\cF, \alpha_s^\vee} & -d_\cF
 \end{bmatrix} \\
 =
 \begin{bmatrix}
  -\varphi^i \circ d_\cF & \BsR \varphi                       & \\
                         & \BsBs (\varphi \circ d_\cF)        & \RBs \varphi \\
                         & -\BsR (\nu_{\varphi, \alpha_s^\vee} \circ d_\cF) - \BsR (\varphi \circ \mu_{\cF, \alpha_s^\vee}) & -\varphi \circ d_\cF
 \end{bmatrix}.
\end{multline*}
The $(1, 1)$ and $(3, 3)$ entries of these two matrices agree because $\varphi$ is a pseudo map. The $(2, 2)$ entries agree by \eqref{eq:wall-crossing-key-computation-morphisms}. The $(3, 2)$ entries agree by Lemma~\ref{lem:mu-and-nu}(2).

Given pseudo maps $\varphi\colon \cF \to \cG$ and $\psi\colon \cG \to \cH$, a similar direct computation using Lemma~\ref{lem:mu-and-nu}(3) shows that $\p\xi_s\psi \circ \p\xi_s\varphi = \p\xi_s(\psi \circ \varphi)$. Thus $\p\xi_s$ is a functor.

\subsubsection*{Step 2: $q \circ \p\xi_s$ factors through $P$}
Given $\psi \in \cHOM^0(\cF, \cG)$ of degree $-2$ and $\beta \in \fm$, we must show that the map $(\p\xi_s)(\beta\psi)\colon \p\xi_s\cF \to \p\xi_s\cG$ of complexes in $\sC_{(T)}$ is nullhomotopic. We claim that a homotopy is given by
\[
 h^i =
 \begin{bmatrix}
  &                                                  &  \\
  &                                                  &  \\
  & \BsR \langle \alpha_s^\vee, \beta \rangle \psi^i &
 \end{bmatrix}
 \colon (\p\xi_s\cF)^i \to (\p\xi_s\cG)^{i-1},
\]
where $\BsR \langle \alpha_s^\vee, \beta \rangle \psi^i\colon B_s\cF^i \to R\cG^i\{-1\}$ is in the $(3, 2)$ entry (see Figure~\ref{fig:homotopy}).

\newsavebox\boxfive
\savebox\boxfive{%
 $\left[
  \begin{smallmatrix}
   -d_\cG^{i-2} & \BsR 1_{\cG^{i-1}}                    & \\
                & \BsBs d_\cG^{i-1}                     & \RBs 1_{\cG^i} \\
                & - \BsR \mu_{\cG, \alpha_s^\vee}^{i-1} & -d_\cG^i
  \end{smallmatrix}
  \right]$
}
\newsavebox\boxsix
\savebox\boxsix{%
 $\left[
  \begin{smallmatrix}
   -d_\cF^{i-1} & \BsR 1_{\cF^i}                    & \\
                & \BsBs d_\cF^i                     & \RBs 1_{\cF^{i+1}} \\
                & - \BsR \mu_{\cF, \alpha_s^\vee}^i & -d_\cF^{i+1}
  \end{smallmatrix}
  \right]$
}

\begin{figure}
 \[
 \scalebox{0.75}{
 \begin{tikzcd}[column sep=small, ampersand replacement=\&]
   \cdots \ar[r]
   \&
   {\begin{matrix} R\cG^{i-2}\{1\} \\ B_s\cG^{i-1} \\ R\cG^i\{-1\} \end{matrix}} \ar[rrrrrrrrr, "\usebox\boxfive"]
   \&\&\&\&\&\&\&\&\&
   {\begin{matrix} R\cG^{i-1}\{1\} \\ B_s\cG^i \\ R\cG^{i+1}\{-1\} \end{matrix}} \ar[rrrrrrrrr]
   \&\&\&\&\&\&\&\&\&
   {\begin{matrix} R\cG^i\{1+1\} \\ B_s\cG^{i+1} \\ R\cG^{i+2}\{-1\} \end{matrix}} \ar[r]
   \&
   \cdots \\
   \cdots \ar[r]
   \&
   {\begin{matrix} R\cF^{i-2}\{1\} \\ B_s\cF^{i-1} \\ R\cF^i\{-1\} \end{matrix}} \ar[rrrrrrrrr] \ar[u]
   \&\&\&\&\&\&\&\&\&
   {\begin{matrix} R\cF^{i-1}\{1\} \\ B_s\cF^i \\ R\cF^{i+1}\{-1\} \end{matrix}} \ar[rrrrrrrrr, "\usebox\boxsix"'] \ar[u, "(\p\xi_s)(\beta\varphi)^i"] \ar[lllllllllu, dashed, "h^i"]
   \&\&\&\&\&\&\&\&\&
   {\begin{matrix} R\cF^i\{1\} \\ B_s\cF^{i+1} \\ R\cF^{i+2}\{-1\} \end{matrix}} \ar[r] \ar[u] \ar[lllllllllu, dashed, "h^{i+1}"]
   \&
   \cdots
 \end{tikzcd}
 }
 \]
 \caption{Homotopy $h$ in Step 2} \label{fig:homotopy}
\end{figure}
Indeed, by direct computation,
\[
 (dh)^i = d_{\p\xi_s\cG}^{i-1} \circ h^i + h^{i+1} \circ d_{\p\xi_s\cF}^i
 =
\begin{bmatrix}
  &                                                                                             & \\
  & \BsRBs \langle \alpha_s^\vee, \beta \rangle \psi^i                                          & \\
  & \BsR \langle \alpha_s^\vee, \beta \rangle (d_\cG^i \circ \psi^i - \psi^{i+1} \circ d_\cF^i) &
 \end{bmatrix}.
\]
By \eqref{eq:polynomial-forcing}, $\BsRBs \langle \alpha_s^\vee, \beta \rangle \psi^i \equiv \BsBs \beta\psi^i \bmod \fm$, and by Lemma~\ref{lem:monodromy-module}(2),
\begin{multline*}
 \nu_{\beta\psi, \alpha_s^\vee} = \Phi_{\cHOM^0(\cF, \cG), \alpha_s^\vee}(d(\beta\psi)) = \Phi_{\cHOM^0(\cF, \cG), \alpha_s^\vee}(\beta(d\psi)) \\
 = \langle \alpha_s^\vee, \beta \rangle (\z(d\psi)) = \langle \alpha_s^\vee, \beta \rangle (d_\cG \circ \psi - \psi \circ d_\cF).
\end{multline*}
Thus $(dh)^i = (\p\xi_s)(\beta\psi)^i$, as desired.

\subsubsection*{Step 3: $q \circ \p\xi_s$ factors through $Q \circ P$}
Since the quotient map $\cHOM^{-1}(\cF, \cG) \to \z\cHOM^{-1}(\cF, \cG)$ is surjective, it remains to show that for any $h \in \cHOM^{-1}(\cF, \cG)$, the map $\p\xi_s(dh)\colon \p\xi_s\cF \to \p\xi_s\cG$ of complexes in $\sC_{(T)}$ is nullhomotopic. We claim that a homotopy is given by
\[
 H^i =
 \begin{bmatrix}
  -h^{i-1} &           & \\
           & \BsBs h^i & \\
           &           & -h^{i+1}
 \end{bmatrix}
 \colon (\p\xi_s\cF)^i \to (\p\xi_s\cG)^{i-1}.
\]
Indeed,
\begin{multline*}
 (dH)^i = d_{\p\xi_s\cG}^{i-1} \circ H^i + H^{i+1} \circ d_{\p\xi_s\cF}^i \\
 =
 \begin{bmatrix}
  (dh)^{i-1} &                                                                                                 & \\
             & \BsBs (dh)^i                                                                                    & \\
             & -\BsR (\mu_{\cG, \alpha_s^\vee}^{i-2} \circ h^i - h^{i+2} \circ \mu_{\cF, \alpha_s^\vee}^{i-1}) & (dh)^{i+1}
 \end{bmatrix},
\end{multline*}
and by Lemma~\ref{lem:monodromy-cHOM}(2),
\begin{multline*}
 \mu_{dh, \alpha_s^\vee} = \Phi_{\alpha_s^\vee}(ddh) = \Phi_{\alpha_s^\vee}(d_\cG \circ d_\cG \circ h - h \circ d_\cF \circ d_\cF) \\
 = \Phi_{\alpha_s^\vee}(d_\cG \circ d_\cG) \circ \z h - \z h \circ \Phi_{\alpha_s^\vee}(d_\cF \circ d_\cF) = \mu_{\cG, \alpha_s^\vee} \circ \z h - \z h \circ \mu_{\cF, \alpha_s^\vee},
\end{multline*}
so $(dH)^i = (\p\xi_s(dh))^i$.

This concludes the construction of $\xi_s$.

\subsection{Proof of Proposition~\ref{prop:wall-crossing}}
\label{ss:wall-crossing-proof}

(1) Since $P$, $Q$, $q$, and $R_K^{-1}$ commute with shifts on the nose, it suffices to prove the claim for $\p\xi_s$. This is a direct computation.
 
(2) By the definition of the triangulated structure on $\z\PKb\sC_T$ and the construction of $\xi_s$, it suffices to show that $q \circ \p\xi_s$ sends a standard triangle to a distinguished triangle. In fact, given a pseudo map $\varphi\colon \cF \to \cG$, we claim that there is an isomorphism
\[
 \begin{tikzcd}[column sep=large]
  \p\xi_s\cF \ar[r, "\p\xi_s\varphi"] \ar[d, equal] & \p\xi_s\cG \ar[r, "\p\xi_s\alpha(\varphi)"] \ar[d, equal] & \p\xi_s C(\varphi) \ar[r, "\p\xi_s\beta(\varphi)"] \ar[d, "\gamma", "\wr"'] & \p\xi_s\cF[1] \ar[d, equal] \\
  \p\xi_s\cF \ar[r, "\p\xi_s\varphi"] & \p\xi_s\cG \ar[r, "\p\xi_s\alpha(\varphi)"] & C(\p\xi_s\varphi) \ar[r, "\p\xi_s\beta(\varphi)"] & \p\xi_s\cF[1]
 \end{tikzcd}
\]
of triangles even in $\Chb\sC_{(T)}$. Here, $\gamma$ is given by the evident isomorphism between
\begin{multline*}
 (\p\xi_s C(\varphi))^i = C(\varphi)^{i-1}\{1\} \oplus C(\varphi)^i \oplus C(\varphi)^{i+1}\{-1\} \\
 = (\cF^i \oplus \cG^{i-1})\{1\} \oplus B_s(\cF^{i+1} \oplus \cG^i) \oplus (\cF^{i+2} \oplus \cG^{i+1})\{-1\}
\end{multline*}
and
\begin{multline*}
 C(\p\xi_s\varphi) = (\p\xi_s\cF)^{i+1} \oplus (\p\xi_s\cG)^i \\
 = (\cF^i\{1\} \oplus B_s\cF^{i+1} \oplus \cF^{i+2}\{-1\}) \oplus (\cG^{i-1}\{1\} \oplus B_s\cG^i \oplus \cG^{i+1}\{-1\}).
\end{multline*}
The only claim that is not clear is that $\gamma$ is a map of complexes. A direct computation writing out $d_{C(\p\xi_s\varphi)}^i$ and $d_{\p\xi_s C(\varphi)}^i$ as 6-by-6 matrices, together with the following lemma, shows that they are indeed identified by $\gamma^i$.
\begin{lem}
 Let $\varphi\colon \cF \to \cG$ be a pseudo map of pseudo complexes. Then
 \begin{align*}
  \mu_{C(\varphi), X}^i = \begin{bmatrix} \mu_{\cF, X}^{i+1} & \\ (-1)^i \nu_{\varphi, X}^{i+1} & \mu_{\cG, X}^i \end{bmatrix}, \quad
  \nu_{\alpha(\varphi), X} = 0, \quad \nu_{\beta(\varphi), X} = 0
 \end{align*}
 for all $X \in \fh$ and $i \in \Z$.
\end{lem}
\begin{proof}
 The first relation follows from calculating $d_{C(\varphi)} \circ d_{C(\varphi)} $. The second and third relations hold because $\alpha(\varphi)$ and $\beta(\varphi)$ commute on the nose (not just modulo $\fm$) with the pseudo differentials. 
\end{proof}

(3) This is clear from the construction.

(4) Choose a basis $\{X_1, \ldots, X_r\}$ of $\fh$ with $X_1 = \alpha_s^\vee$, $s(X_i) = X_i$ for $i > 1$. Then
\[
 R^s = \bk[(\alpha_s^\vee)^2, X_2, \ldots, X_r],
\]
so it suffices to consider $f = (\alpha_s^\vee)^2$ and $f = X \in \fh$ with $\la \alpha_s, X \ra = 0$. For these cases, it is straightforward to verify the statement directly from the definitions.

(5) Let $\cG$ be a pseudo complex. Choose a basis $X_1, \ldots, X_r$ of $\fh$, and write $d_\cG \circ d_\cG = \sum_j X_j^* \widetilde{\mu}_{\cG, X_j}$. Since $F$ is $R$-linear, $F(d_\cG \circ d_\cG) = \sum_j X_j^* F(\widetilde{\mu}_{\cG, X_j})$. It follows that $\mu_{F\cG, X} = F(\mu_{\cG, X})\colon F\cG \to F\cG\{2\}$ for any $X \in \fh$. The rest of the argument is straightforward.

This concludes the proof of Proposition~\ref{prop:wall-crossing}.

\subsection{Generating tilting objects}
\label{ss:tilting}
Let $s, t \in S$. Applying Proposition~\ref{prop:wall-crossing} with $\sC_T = \Parity(\BGB)$ and $\sC_T = \Parity(\BGPt)$, we obtain exact functors
\[
 \xi_s\colon \Dmix(\UGB) \to \Dmix(\UGB), \quad \xi_s\colon \Dmix(\UGPt) \to \Dmix(\UGPt),
\]
which we call \emph{wall-crossing functors}. It follows follow from Proposition~\ref{prop:wall-crossing}(5) that
\begin{equation} \label{eq:wall-crossing-push-pull}
 \xi_s \circ \pi_t^* \cong \pi_t^* \circ \xi_s, \quad \xi_s \circ \pi_{t*} \cong \pi_{t*} \circ \xi_s,
\end{equation}
and that for $\cF \in \Dmix(\UGB)$ and $\cG \in \Dmix(\BGB)$, we have
\begin{equation} \label{eq:wall-crossing-conv}
 (\xi_s\cF) \ast \cG \cong \xi_s(\cF \ast \cG).
\end{equation}

The following result is an analogue of the mixed version of~\cite[Lemma~5.21]{AR-I}, and is proved in the same way.
\begin{lem} \label{lem:translation-on-standard-and-costandard}
 \begin{enumerate}
  \item For all $w \in W$, $\xi_s\Delta_w$ is perverse. It admits a standard filtration with associated graded $\Delta_{sw} \oplus \Delta_w \la1\ra$ if $sw > w$ and $\Delta_{sw} \oplus \Delta_w\la-1\ra$ if $sw < w$.
  \item For all $w \in W$, $\xi_s\nabla_w$ is perverse. It admits a standard filtration with associated graded $\nabla_{sw} \oplus \nabla_w \la-1\ra$ if $sw > w$ and $\nabla_{sw} \oplus \nabla_w\la1\ra$ if $sw < w$.
 \end{enumerate}
\end{lem}
For any expression $\uw = s_1 \ldots s_k$, define the \emph{Bott--Samelson tilting object}
\[
 \cT_\uw := \xi_{s_1} \cdots \xi_{s_k}(\delta).
\]
Lemma~\ref{lem:translation-on-standard-and-costandard} shows that $\cT_\uw \in \Tmix(\UGB)$, and also implies the following Bott--Samelson characterization of indecomposable tilting objects, analogous to the one for Soergel modules (Proposition~\ref{prop:sbim-bs-characterization}).
\begin{prop} \label{prop:tilting-bs-characterization}
 For any reduced expression $\uw$ of $w \in W$, $\cT_w$ can be identified with the unique indecomposable direct summand of $\cT_\uw$ that does not occur as a direct summand of $\cT_\ux$ for any expression $\ux$ with $\ell(\ux) < \ell(\uw)$.
\end{prop}

We also need a tilting analogue of the Soergel Hom formula \eqref{eq:Soergel-module-hom-formula}. In general, for a graded highest weight category $(\cA, \la1\ra)$ indexed by $(\scS, \le)$, let $\cF_\Delta$ (resp.~$\cF_\nabla$) denote the full subcategory consisting of standardly (resp.~costandardly) filtered objects. For $X \in \cF_\Delta$, let $(X \colon \Delta_s\la n \ra)$ denote the multiplicity of $\Delta_s\la n \ra$ in any standard filtration of $X$. For $Y \in \cF_\nabla$, similarly write $(Y \colon \nabla_s\la n \ra)$. Let $X \in \cF_\Delta$ and $Y \in \cF_\nabla$. Since $\Ext^k(\Delta_s, \nabla_t\la n \ra) = 0$ for any $s, t \in \scS$, $n \in \Z$, $k > 0$, we get
\[
 \dim\Hom_\cA(X, Y) = \sum_{s \in \scS, n \in \Z} (X : \Delta_s\la n \ra)(Y : \nabla_s\la n \ra)
\]
by inducting on the length of a standard (resp.~costandard) filtration of $X$ (resp.~$Y$). It follows that
\begin{align} \label{eq:ghw-gdim}
 \gdim \bigoplus_{n \in \Z} \Hom_\cA(X, Y\la n \ra) = \sum_{\substack{s \in \scS \\ n_1, n_2 \in \Z}} (X : \Delta_s\la n_1 \ra)(Y : \nabla_s \la n_2 \ra)v^{n_1 - n_2}.
\end{align}

Now consider $\cA = \Pmix(\UGB)$. Recall the Hecke algebra $\cH_W$ (see~\S\ref{ss:coxeter} for our normalization) and the pairing $\la -, - \ra$ (see~\S\ref{sss:hom-formula-equivariant-formality}). Define
\begin{align*}
 \ch_\Delta&\colon \Ob(\cF_\Delta) \to \cH_W\colon \cF \mapsto \sum_{w \in W, n \in \Z} (\cF : \Delta_w\la n \ra)v^nH_w, \\
 \ch_\nabla&\colon \Ob(\cF_\nabla) \to \cH_W\colon \cG \mapsto \sum_{w \in W, n \in \Z} (\cG : \nabla_w\la n \ra)v^{-n}H_w.
\end{align*}
Then \eqref{eq:ghw-gdim} may be restated as follows: for $\cF \in \cF_\Delta$ and $\cG \in \cF_\nabla$, we have
\begin{align} \label{eq:ghw-hom}
 \gdim \bigoplus_{n \in \Z} \Hom(\cF, \cG\la n \ra) = \la\ch_\Delta(\cF), \ch_\nabla(\cG)\ra.
\end{align}
Lemma~\ref{lem:translation-on-standard-and-costandard} implies that each $\xi_s$ restricts to an endofunctor on $\cF_\Delta$ (resp.~$\cF_\nabla$), and for $\cF \in \cF_\Delta$ and $\cG \in \cF_\nabla$, we have
\begin{equation} \label{eq:wall-crossing-ch}
 \ch_\Delta(\xi_s\cF) = \uH_s \ch_\Delta(\cF), \quad \ch_\nabla(\xi_s\cG) = \uH_s \ch_\nabla(\cG).
\end{equation}
Given expressions $\ux, \uy$, it follows from \eqref{eq:ghw-hom} and \eqref{eq:wall-crossing-ch} that
\begin{equation} \label{eq:tilting-hom-formula}
 \gdim \bigoplus_{n \in \Z} \Hom(\cT_\ux, \cT_\uy\la n \ra) = (\ch_\Delta(\cT_\ux), \ch_\nabla(\cT_\uy)) = \la\uH_\ux, \uH_\uy\ra.
\end{equation}

\section{Koszul duality}
\label{sec:koszul-duality}

Let $(W, \fh)$ be a reflection faithful realization. In this section, we assume in addition that $W$ is finite. Denote the longest element of $W$ by $w_0$.

\subsection{Preliminaries}
\label{ss:koszul-duality-preliminaries}
We collect a few results about $\Pmix(\UGB)$.

As in~\cite{AR-II}, the following result may be proved by imitating the argument of~\cite[\S2.1]{BBM} or \cite[Lemma~4.4.7]{BY}.
\begin{lem}[cf.~\cite{AR-II}, Lemma~4.9] \label{lem:delta-in-standard-and-costandard}
 Let $w \in W$.
 \begin{enumerate}
  \item There exists an embedding $\delta \la-\ell(w)\ra \hookrightarrow \Delta_w$ whose cokernel has no composition factor of the form $\delta\la n \ra$.
  \item There exists a surjection $\nabla_w \twoheadrightarrow \delta\la \ell(w) \ra$ whose kernel has no composition factor of the form $\delta\la n \ra$.
 \end{enumerate}
\end{lem}

Fix once and for all a projective cover $\pi\colon \cP_e \to \delta$ of the skyscraper.
\begin{lem} \label{lem:standard-filtration-of-projective-cover-of-delta}
 Let $w \in W$. We have
 \[
  (\cP_e : \Delta_w\la n \ra) = \begin{cases}
                                 1 &\mbox{if } n = -\ell(w); \\
                                 0 &\mbox{otherwise.}
                                \end{cases}
 \]
\end{lem}
\begin{proof}
 This follows from graded BGG reciprocity~\cite[Theorem~A.3]{AR-II} and\\Lemma~\ref{lem:delta-in-standard-and-costandard}(2).
\end{proof}

\begin{lem} \label{lem:delta-in-translated-projective-cover-of-delta}
 Let $s \in S$. Then $[\xi_s\cP_e\la-1\ra : \delta] = 1$.
\end{lem}
\begin{proof}
 Use Lemma~\ref{lem:standard-filtration-of-projective-cover-of-delta} and Lemma~\ref{lem:translation-on-standard-and-costandard} to find the associated graded of the standard filtration of $\xi_s\cP_e\la-1\ra$, then use Lemma~\ref{lem:delta-in-standard-and-costandard}.
\end{proof}

\subsection{$\bV$ functor}
\label{ss:V-functor}
Define $\bV$ as the composition
\[
 \Dmix(\UGB) \xrightarrow{\bigoplus_{n \in \Z} \Hom(\cP_e, (-)\la n \ra)} \gmod\lh \left(\bigoplus_{n \in \Z} \Hom(\cP_e, \cP_e\la n \ra)\right) \xrightarrow{\mu^*_{\cP_e}} R^\vee\lh\gmod,
\]
where $\mu^*_{\cP_e}$ denotes pullback under the monodromy map
\[
 \mu_{\cP_e}\colon R^\vee \to \bigoplus_{n \in \Z} \Hom(\cP_e, \cP_e\la n \ra)
\]
constructed in~\S\ref{sec:monodromy}. Recall the functors
\[
 \theta_s := B_s^\vee \otimes_{R^\vee} (-)\colon R^\vee\lh\gmod \to R^\vee\lh\gmod, \quad B_s^\vee := (R^\vee) \otimes_{(R^\vee)^s} R^\vee \{1\}
\]
``generating'' $\Parity(\BGUvee)$.
\begin{prop} \label{prop:V-functor-properties}
 The $\bV$ functor is cohomological (transforms a distinguished triangle into a long exact sequence) and satisfies $\bV \circ \la1\ra = \{1\} \circ \bV$. Moreover, for any $s \in S$, there is a natural isomorphism $\theta_s \circ \bV \cong \bV \circ\xi_s$.
\end{prop}
The following proof uses subregular Soergel theory (see~\S\ref{sss:ssbim}).
\begin{proof}
 The first two claims are clear from the construction. For the last claim, let $s \in S$. Fix a nonzero morphism $\can_s\colon \cP_e \to \xi_s\cP_e\la-1\ra$ in $\Pmix(\UGB)$, unique up to scalar by Lemma~\ref{lem:delta-in-translated-projective-cover-of-delta}. For each $\cF \in \Dmix(\UGB)$, define the graded $\bk$-linear map
 \[
  \gamma'_\cF\colon \bV(\cF)\{1\} \to \bV(\xi_s\cF)
 \]
 on homogeneous elements by
 \begin{align*}
  (\cP_e \xrightarrow{\varphi} \cF\la n+1 \ra) \mapsto (\cP_e \xrightarrow{\can_s} \xi_s\cP_e \la-1\ra \xrightarrow{\xi_s\varphi \la-1\ra} \xi_s\cF \la n \ra).
 \end{align*}
 It follows from Lemma~\ref{lem:mu-as-graded-center} and Proposition~\ref{prop:wall-crossing}(4) that $\gamma'_\cF$ is $(R^\vee)^s$-linear, so it induces a natural transformation
 \[
  \gamma\colon \theta_s \circ \bV \to \bV \circ \xi_s.
 \]
 
 We claim that this is an equivalence. Since $\bV$ is cohomological and $\theta_s$ is exact (because $R^\vee$ is free over $(R^\vee)^s$), $\theta_s \circ \bV$ is cohomological. Similarly, since $\xi_s$ is exact, $\bV \circ \xi_s$ is cohomological. Thus by the five lemma, it suffices to show that
 \begin{equation} \label{eq:gamma_IC}
  \gamma_{\IC_w}\colon R^\vee \otimes_{(R^\vee)^s} \bV(\IC_w)\{1\} \to \bV(\xi_s\IC_w)
 \end{equation}
 is an isomorphism for all $w \in W$.

 First suppose $w \neq e$, so $\bV(\IC_w) = 0$. We claim that $\bV(\xi_s\IC_w) = 0$. There is some $t \in S$ (which may be $s$) with $wt < w$, and $\IC_w \cong \pi^{t*}\IC_{\overline{w}}\{1\}$, where $\overline{w}$ is the image of $w$ in $W/\{1, t\}$. So by~\eqref{eq:wall-crossing-push-pull}, $\xi_s\IC_w \cong \xi_s\pi^{t*}\IC_{\overline{w}}\{1\} \cong \pi^{t*}\xi_s\IC_{\overline{w}}\{1\}$. Since $\pi^{t*}\{1\}$ is perverse t-exact, this shows that no twist of $\delta$ can appear as a composition factor of $\xi_s\IC_w$, and the claim follows.

 Now let $w = e$. Then as graded $\bk$-vector spaces, both sides of~\eqref{eq:gamma_IC} are isomorphic to $\bk\{1\} \oplus \bk\{-1\}$. From the monodromy of $\xi_s\delta = \cT_s$, we know that the action of $\alpha_s^\vee$ on $\bV(\xi_s\delta)$ maps the degree $-1$ part isomorphically to the degree $1$ part. Hence to show that $\gamma_\delta$ is an isomorphism, it is enough to check it in degree $-1$, i.e.~that $\gamma_\delta(1 \otimes \pi) = \xi_s\pi\la-1\ra  \circ \can_s$ is nonzero. Since $\xi_s$ is t-exact, $\xi_s\pi\la-1\ra$ is an epimorphism, so the map
 \[
  \xi_s\pi\la-1\ra \circ -\colon \Hom(\cP_e, \xi_s\cP_e\la-1\ra) \to \Hom(\cP_e, \xi_s\delta\la-1\ra)
 \]
 is surjective. But the right hand side is one-dimensional, and so is the left hand side by Lemma~\ref{lem:delta-in-translated-projective-cover-of-delta}, so this is an isomorphism. In particular, $\xi_s\pi\la-1\ra  \circ \can_s \neq 0$.
\end{proof}

\subsection{Proof of the main result}
\label{ss:main-result-proof}
We begin with the following analogue of~\cite[Proposition~5.3]{AR-II}.
\begin{prop} \label{prop:tilting-parity-duality}
 The $\bV$ functor restricts to an equivalence of additive categories
 \[
  \nu\colon \Tmix(\UGB) \simto \Parity(\BGUvee)
 \]
 satisfying $\nu \circ \la1\ra \cong \{1\} \circ \nu$ and $\nu(\cT_w) \cong \cE^\vee_w$.
\end{prop}
\begin{proof}
 The claim about the interaction with shifts follows from the corresponding claim in Proposition~\ref{prop:V-functor-properties}. The last claim in Proposition~\ref{prop:V-functor-properties} implies that $\bV(\cT_\uw) \cong \cE^\vee_\uw$ for any expression $\uw$. It then follows from the Bott--Samelson characterization of indecomposable parity complexes (Proposition~\ref{prop:sbim-bs-characterization}) and indecomposable tilting objects (Proposition~\ref{prop:tilting-bs-characterization}) that $\bV$ restricts to a functor $\nu$ as claimed, and that $\nu(\cT_w) \cong \cE^\vee_w$.

 By the argument of~\cite[\S2.1]{BBM}, it follows from Lemma~\ref{lem:delta-in-standard-and-costandard} that $\nu$ is faithful. It therefore suffices to compare dimensions of Hom spaces between Bott--Samelson objects. These agree by the Soergel Hom formula~\eqref{eq:Soergel-module-hom-formula} and its tilting analogue~\eqref{eq:tilting-hom-formula}.
\end{proof}

We are now ready to prove our main result.
\begin{proof}[Proof of Theorem~\ref{thm:mkd}]
 Since $\Pmix(\UGB)$ is graded highest weight, the natural functors
 \begin{equation} \label{eq:tilt-perv-der-eq}
  \Kb\Tmix(\UGB) \to \Db\Pmix(\UGB) \to \Dmix(\UGB)
 \end{equation}
 are equivalences by \cite[Lemma~B.5]{AR-I} (cf.~\cite[Lemma~3.15]{AR-II}). Define $\kappa$ as the composition
 \begin{multline*}
 \Dmix(\UGB) = \Kb\Parity(\UGB) \xrightarrow[\sim]{\Kb\nu} \Kb\Tmix(\BGUvee) \\
 \xrightarrow[\sim]{\text{\eqref{eq:tilt-perv-der-eq}}} \Dmix(\BGUvee).
 \end{multline*}
 The claims about the interaction with shifts are clear. Proposition~\ref{prop:tilting-parity-duality} implies $\kappa(\cT_w) \cong \cE^\vee_w$. This is enough to determine $\kappa(\Delta_w)$ and $\kappa(\nabla_w)$ as in the proof of~\cite[Lemma~5.2]{AR-II}.
 
 It remains to show that $\kappa(\cE_w) \cong \cT^\vee_w$. Consider the functor
 \[
  \kappa^\vee\colon \Dmix(\BGUvee) \simto \Dmix(\UGB)
 \]
 defined in the same way for $G^\vee$, so all but the last claim is also known for $\kappa^\vee$. Since $\cT^\vee_w$ is a successive extension of various $\Delta^\vee_x\la n \ra$ (resp.~$\nabla^\vee_x\la n \ra$), we may apply $\kappa \circ \kappa^\vee$ to conclude the same for $\kappa(\cE_w)$. Hence $\kappa(\cE_w)$ is perverse. Repeating this argument, we deduce that $\kappa(\cE_w)$ is tilting. Since $\cE_w$ is indecomposable, so is $\kappa(\cE_w)$. By inducting on $w$ as in the argument of~\cite[Lemma~5.2]{AR-II}, we see that the support condition and the normalization of $\cE_w$ implies the same for $\kappa(\cE_w)$. These conditions characterize $\cT^\vee_w$.
\end{proof}

\bibliographystyle{alpha}
\bibliography{refs.bib}

\end{document}